\documentclass[reqno, 12pt]{amsart}
\pdfoutput=1
\makeatletter
\let\origsection=\section \def\section{\@ifstar{\origsection*}{\mysection}} 
\def\mysection{\@startsection{section}{1}\z@{.7\linespacing\@plus\linespacing}{.5\linespacing}{\normalfont\scshape\centering\S}}
\makeatother        

\usepackage{amsmath,amssymb,amsthm}
\usepackage{mathrsfs}
\usepackage{mathabx}\changenotsign
\usepackage{wasysym}
\usepackage{dsfont}
 
\usepackage{xcolor}
\usepackage[backref]{hyperref}
\hypersetup{
    colorlinks,
    linkcolor={red!60!black},
    citecolor={green!60!black},
    urlcolor={blue!60!black}
}

\usepackage[open,openlevel=2,atend]{bookmark}

\usepackage[abbrev,msc-links,backrefs]{amsrefs} 
\usepackage{doi}

\renewcommand{\PrintDOI}[1]{\doi{#1}}

\usepackage[T1]{fontenc}
\usepackage{lmodern}
\usepackage[babel]{microtype}
\usepackage[english]{babel}

\usepackage{geometry}
\numberwithin{equation}{section}
\numberwithin{figure}{section}

\usepackage{scalerel}

\makeatletter
\newcommand{\overrighharpoonup}[1]{\ThisStyle{ \vbox {\m@th\ialign{##\crcr
 \rightharpoonupfill \crcr
 \noalign{\kern-\p@\nointerlineskip}
 $\hfil\SavedStyle#1\hfil$\crcr}}}}

\def\rightharpoonupfill{$\SavedStyle\m@th\mkern+0.8mu\cleaders\hbox{$\shortbar\mkern-4mu$}\hfill\rightharpoonuptip\mkern+0.8mu$}

\def\rightharpoonuptip{ \raisebox{\z@}[2pt][1pt]{\scalebox{0.55}{$\SavedStyle\rightharpoonup$}}}

\def\shortbar{ \smash{\scalebox{0.55}{$\SavedStyle\relbar$}}}
\makeatother

\usepackage{enumitem}
\def\rmlabel{\upshape({\itshape \roman*\,})}

\def\alabel{\upshape({\itshape \alph*\,})}

\def\nlabel{\upshape({\itshape \arabic*\,})}

\makeatletter
\def\greek#1{\expandafter\@greek\csname c@#1\endcsname}
\def\Greek#1{\expandafter\@Greek\csname c@#1\endcsname}
\def\@greek#1{\ifcase#1
	\or $\alpha$\or $\beta$\or $\gamma$\or $\delta$\or $\epsilon$\or $\zeta$\or $\eta$\or 
	$\theta$\or $\iota$\or $\kappa$\or $\lambda$	\or $\mu$\or $\nu$\or $\xi$\or $o$\or 
	$\pi$\or $\rho$\or $\sigma$\or $\tau$\or $\upsilon$\or $\phi$\or $\chi$\or $\psi$\or 
	$\omega$\fi}
\def\@Greek#1{\ifcase#1
	\or $\mathrm{A}$\or $\mathrm{B}$\or $\Gamma$\or $\Delta$\or $\mathrm{E}$\or $\mathrm{Z}$\or 
	$\mathrm{H}$\or $\Theta$\or $\mathrm{I}$\or $\mathrm{K}$\or $\Lambda$\or $\mathrm{M}$\or 
	$\mathrm{N}$\or $\Xi$\or $\mathrm{O}$\or $\Pi$\or $\mathrm{P}$\or $\Sigma$\or 
	$\mathrm{T}$\or $\mathrm{Y}$\or $\Phi$\or $\mathrm{X}$\or $\Psi$\or $\Omega$\fi}

\AddEnumerateCounter{\greek}{\@greek}{24}
\AddEnumerateCounter{\Greek}{\@Greek}{12}

\makeatother

\let\polishlcross=\l
\def\l{\ifmmode\ell\else\polishlcross\fi}

\let\sm=\setminus

\makeatletter
\def\moverlay{\mathpalette\mov@rlay}
\def\mov@rlay#1#2{\leavevmode\vtop{   \baselineskip\z@skip \lineskiplimit-\maxdimen
   \ialign{\hfil$\m@th#1##$\hfil\cr#2\crcr}}}
\newcommand{\charfusion}[3][\mathord]{
    #1{\ifx#1\mathop\vphantom{#2}\fi
        \mathpalette\mov@rlay{#2\cr#3}
      }
    \ifx#1\mathop\expandafter\displaylimits\fi}
\makeatother

\newcommand{\dcup}{\charfusion[\mathbin]{\cup}{\cdot}}
\newcommand{\bigdcup}{\charfusion[\mathop]{\bigcup}{\cdot}}

\DeclareFontFamily{U}  {MnSymbolC}{}
\DeclareSymbolFont{MnSyC}         {U}  {MnSymbolC}{m}{n}
\DeclareFontShape{U}{MnSymbolC}{m}{n}{
    <-6>  MnSymbolC5
   <6-7>  MnSymbolC6
   <7-8>  MnSymbolC7
   <8-9>  MnSymbolC8
   <9-10> MnSymbolC9
  <10-12> MnSymbolC10
  <12->   MnSymbolC12}{}
\DeclareMathSymbol{\powerset}{\mathord}{MnSyC}{180}
\DeclareMathSymbol{\leftY}{\mathord}{MnSyC}{42}

\DeclareSymbolFont{symbolsC}{U}{txsyc}{m}{n}
\SetSymbolFont{symbolsC}{bold}{U}{txsyc}{bx}{n}
\DeclareFontSubstitution{U}{txsyc}{m}{n}
\DeclareMathSymbol{\strictif}{\mathrel}{symbolsC}{74}

\let\epsilon=\varepsilon
\let\eps=\epsilon
\let\rho=\varrho

\def\NN{{\mathds N}}

\def\ZZ{{\mathds Z}}

\def\RR{{\mathds R}}

\newcommand{\cX}{\mathcal{X}}
\newcommand{\cY}{\mathcal{Y}}
\newcommand{\cZ}{\mathcal{Z}}

\newcommand{\ccL}{\mathscr{L}}

\def\NN{\mathds{N}}

\def\RR{\mathds{R}}

\def\ZZ{\mathds{Z}}

\def\bx{\mathbf{x}}

\def\olg{\overline{\gamma}}
\def\oleta{\overline{\eta}}

\def\1{\mathbbm{1}}
\def\<{\langle}
\def\>{\rangle}

\theoremstyle{plain}
\newtheorem{thm}{Theorem}[section]
\newtheorem{theorem}[thm]{Theorem}
\newtheorem{question}[thm]{Question}\newtheorem{lemma}[thm]{Lemma}
\newtheorem{proposition}[thm]{Proposition}

\newtheorem{fact}[thm]{Fact}
\newtheorem{claim}[thm]{Claim}

\theoremstyle{definition}
\newtheorem{remark}[thm]{Remark}
\newtheorem{dfn}[thm]{Definition}

\let\phi=\varphi
\let\theta=\vartheta
\let\lra=\longrightarrow
\let\vn=\varnothing
\usepackage{yfonts}

\DeclareMathOperator{\AP}{AP}

\DeclareMathOperator{\vdW}{vdW}
\DeclareMathOperator{\Sz}{Sz}

\newcommand{\conc}{\charfusion[\mathbin]{+}{\times}}

\usepackage{tikz}
\usetikzlibrary{shapes.misc,calc,intersections,patterns,decorations.pathreplacing, calligraphy}
\usetikzlibrary{decorations.markings}
\usetikzlibrary{calc,positioning,decorations.pathmorphing,decorations.pathreplacing}
\usetikzlibrary{hobby,patterns}

\usepackage{multicol}
\usepackage{subcaption}
\begin{document}

\title{Colouring versus density in integers and Hales-Jewett cubes}

\author{Christian Reiher}
\address{Fachbereich Mathematik, Universit\"{a}t Hamburg, Hamburg, Germany}
\email{christian.reiher@uni-hamburg.de}
\thanks{We acknowledge financial support from the Open Access Publication Fund of Universit\"at Hamburg.}

\author{Vojt\v{e}ch R\"{o}dl}
\address{Department of Mathematics, Emory University, 
    Atlanta, GA, USA}
\email{vrodl@emory.edu}

\author{Marcelo Sales}
\address{Department of Mathematics, University of California, Irvine, CA, USA}
\email{mtsales@uci.edu}
\thanks{The second and third author were supported by NSF grant DMS 1764385,
and the second author was also supported by NSF grant DMS 2300347.}

\begin{abstract}
We construct for every integer $k\geq 3$ and every real $\mu\in(0, \frac{k-1}{k})$ 
a set of integers $X=X(k, \mu)$ which, when coloured with finitely many colours, 
contains a monochromatic $k$-term arithmetic progression, whilst every 
finite $Y\subseteq X$ has a subset $Z\subseteq Y$ of size $|Z|\geq \mu |Y|$ that 
is free of arithmetic progressions of length $k$. This answers a question of 
Erd\H{o}s, Ne\v{s}et\v{r}il, and the second author. 
Moreover, we obtain an analogous multidimensional statement and a 
Hales-Jewett version of this result.  
\end{abstract}

\maketitle

\section{Introduction}\label{sec:intro}

\subsection{Arithmetic progressions}
For every natural number $n$ we set $[n]=\{1,\dots, n\}$. Given a set $X$ and a 
nonnegative integer $k$, we write $X^{(k)}=\{e\subseteq X\colon |e|=k\}$ 
for the collection of all $k$-subsets of $X$. The theorem of van der Waerden 
is one of the earliest results in Ramsey theory. It asserts that every finite 
colouring of the integers yields monochromatic arithmetic progressions of arbitrary
length. More precisely, for positive integers $k$ and $r$ we say that a set of 
integers~$X\subseteq \NN$ has the \textit{van der Waerden property $\vdW(k,r)$} 
if every $r$-colouring of $X$ contains a monochromatic $\AP_k$, i.e., an arithmetic 
progression of length~$k$. Now van der Waerden's theorem~\cite{vdW27} can be 
briefly stated as follows.

\begin{theorem}[van der Waerden]\label{thm:vanderwaerden}
	For all integers $k\geq 3$ and $r\geq 2$ there exists an integer $w=w(k,r)$ 
	such that for every $n\geq w$ the set $[n]$ has the property $\vdW(k,r)$. \qed
\end{theorem}

Solving a long standing conjecture of Erd\H{o}s and Tur\'{a}n~\cite{ET36}, 
Szemer\'{e}di proved the following celebrated result in~\cite{Sz75}.

\begin{theorem}[Szemer\'{e}di]\label{thm:szemeredi}
    For every integer $k\geq 3$ and every real $\delta\in (0,1]$ there exists 
    an integer $n_0=n_0(k,\delta)$ such that for every $n\geq n_0$ the following 
    holds: Every subset $A\subseteq [n]$ of size $|A|\geq \delta n$ contains 
    an $\AP_k$. \qed
\end{theorem}

In other words, Szemer\'{e}di's theorem states that every subset of $\NN$ 
with positive upper density contains arbitrarily long arithmetic progressions. 
This result stimulated a lot of research and today there are 
many proofs using tools from a diverse spectrum of mathematical areas,
including ergodic theory and higher order Fourier analysis 
(see, e.g.,~\cites{F77, FK78, G01, G07, NRS06, RS04, T06}).

Similar to the van der Waerden property $\vdW(k,r)$ one can define a property 
related to Theorem \ref{thm:szemeredi}. For an integer $k\geq 3$ and a 
real~$\delta>0$ we say that a finite set of integers $X\subseteq \NN$ has the 
\textit{Szemer\'{e}di property $\Sz(k,\delta)$} if every subset $Y\subseteq X$
of size $|Y|\geq \delta |X|$ contains an $\AP_k$. So Szemer\'edi's theorem states 
that $[n]$ has the property $\Sz(k, \delta)$ whenever $n\geq n_0(k, \delta)$.

By looking at the densest colour class one sees that for $\delta\le 1/r$ 
the property~$\Sz(k,\delta)$ yields~$\vdW(k,r)$. In this sense, Szemer\'{e}di's 
theorem implies van der Waerden's theorem. One could argue that Szemer\'edi's 
original proof shows that, conversely, van der Waerden's theorem implies his theorem,
but it is safe to say that this direction is a lot deeper. 
Motivated by a famous problem of Pisier~\cite{P83} the following question was considered in~\cites{ENR90, BRSSTW06}.

\begin{question}\label{q:AP_k}
	Is it true that for every $k\geq 3$ there are $\delta>0$ and a set of integers $X$ 	
	such that
	\begin{enumerate}[label=\rmlabel]
   	   \item $X$ has the property $\vdW(k,r)$ for every $r\geq 1$;
    	\item every finite $Y\subseteq X$ fails to have property $\Sz(k,\delta)$?
	\end{enumerate}
\end{question}

The question did also appear in a list of open problems in additive combinatorics
compiled by Croot and Lev (see~\cite{CL07}*{Problem~3.10}).
Notice that a negative answer would show that the properties $\vdW(k,r)$ 
and $\Sz(k,\delta)$ are equivalent. This would, in particular, provide a surprising 
new way of deducing Szemer\'{e}di's theorem from van der Waerden's theorem. 
For this reason, the authors of~\cite{ENR90} conjectured that Question~\ref{q:AP_k} 
has a positive answer. The present work confirms this.

\begin{theorem}\label{thm:main1}
    For every integer $k\geq 3$ and real $\mu\in (0,\frac{k-1}{k})$ there is a 
    set of natural numbers $X=X(k,\mu)\subseteq \NN$ such that
    \begin{enumerate}[label=\rmlabel]
        \item\label{it:14i} for every $r\geq 1$ every $r$-colouring of $X$ contains 
        	a monochromatic $\AP_k$,
        \item\label{it:14ii} but every finite subset $Y\subseteq X$ contains a 
        		subset $Z\subseteq Y$ of size $|Z|\geq \mu|Y|$ without an~$\AP_k$.
    \end{enumerate}
\end{theorem}

We remark that Theorem \ref{thm:main1} does not hold for $\mu>\frac{k-1}{k}$. 
Indeed, every set $X\subseteq \NN$ satisfying condition~\ref{it:14i} must 
contain an $\AP_k$. By taking $Y\subseteq X$ to be an $\AP_k$, we have $|Y|=k$. 
Therefore, the only subset $Z\subseteq Y$ with $|Z|\geq \mu|Y|$ is $Y$ itself. 
So $Y$ has the property~$\Sz(k,\mu)$.

\subsection{Multidimensional version} Gallai and Witt discovered independently 
that van der Waerden's theorem generalises as follows to higher dimensions. 
If $F\subseteq \ZZ^d$ is a finite configuration of $d$-dimensional lattice points 
and $r$ denotes a number of colours, then there is some integer $n=n(F, r)$ such 
that for every $r$-colouring of $[n]^d$ there is a monochromatic homothetic copy
of $F$. This means that there are a vector $v\in\ZZ^d$ and a positive integral 
scaling factor $\lambda$ such that the set $v+\lambda F=\{v+\lambda f\colon f\in F\}$
is contained in $[n]^d$ and all of its points have the same colour. There is also a 
multidimensional version of Szemer\'edi's theorem, first proved by F\"urstenberg 
and Katznelson~\cite{FK78}, which asserts that for every finite 
configuration $F\subseteq \ZZ^d$ and every positive real $\delta$ there exists 
some $n=n(F, \delta)$ such that every set $A\subseteq [n]^d$ of 
size $|A|\ge \delta n^d$ contains a homothetic copy of $F$. 

When formulating a multidimensional version of Theorem~\ref{thm:main1} one can 
make the second clause stronger by forbidding the large subsets $Z$ to contain 
``copies of $F$'' in a sense more general than ``homothetic copies''. 
Given a finite configuration $F\subseteq \ZZ^d$ and a set $Z\subseteq \ZZ^d$
we shall say that $Z$ is {\it $F$-free} if there is no nonzero real scaling 
factor $\lambda$ such that $Z$ contains a congruent copy of $\lambda F$. 

Illustrating the difference between these concepts we consider the sets 
\[
	F=\{(0, 0), (0, 1), (1, 0), (1, 1)\}
	\quad\text{ and } \quad
	F'=\{(1, 0), (-1, 0), (0, 1), (0, -1)\}\,,
\]
both of which are squares in $\ZZ^2$. They fail to be homothetic copies of 
each other, but $F'$ is a congruent copy of $\sqrt{2}F$.

\begin{theorem}\label{thm:main2}
    For every finite configuration $F\subseteq \ZZ^d$ of $|F|=k\ge 3$ points 
    and every real $\mu\in (0,\frac{k-1}{k})$ there is a set of lattice points 
    $X=X(F,\mu)\subseteq \NN^d$ such that
    \begin{enumerate}[label=\rmlabel]
        \item\label{it:15i} for every $r\geq 1$, every $r$-colouring of $X$ 
        		contains a monochromatic homothetic copy of $F$,
        \item\label{it:15ii} but every finite subset $Y\subseteq X$ has an 
        		$F$-free subset $Z\subseteq Y$ of size $|Z|\geq \mu|Y|$.
    \end{enumerate}
\end{theorem}

\subsection{Combinatorial lines}\label{subsec:cl}
Let us now briefly explain another classical result of Ramsey theory, 
the Hales-Jewett theorem. Given integers $k\ge 2$ and $n\ge 1$ 
we refer to $[k]^n$ as the {\it $n$-dimensional Hales-Jewett cube over $[k]$}. 
Identifying $[k]^1$ with $[k]$ we can regard~$[k]$ itself as a one-dimensional cube.
A map $\eta\colon [k]\longrightarrow [k]^n$ is called a {\it combinatorial embedding}
if there exists a partition of $[n]=C\dcup M$ of the coordinate set $[n]$ into a set 
of {\it constant coordinates}~$C$ and a nonempty set of 
{\it moving coordinates}~$M$ such that, writing $\eta(i)=(u_{i1}, \dots, u_{in})$ 
for each $i\in [k]$, we have $u_{1c}=\dots=u_{kc}$ for every $c\in C$ 
and $u_{im}=i$ whenever $i\in [k]$ and $m\in M$. The condition $M\ne\vn$ ensures 
that combinatorial embeddings are injective. 

A {\it combinatorial line} is the image of such a combinatorial 
embedding. For instance, $\{111, 121, 131\}$ is a combinatorial 
line in $[3]^3$, while $\{113, 122, 131\}$ is not. 
Now the result of Hales and Jewett~\cite{HJ63} asserts that given $k\ge 2$ 
and $r\ge 1$ there exists a dimension $n=\mathrm{HJ}(k, r)$
such that for every colouring $f\colon [k]^n\lra [r]$ there exists a monochromatic 
combinatorial line. An even more profound result, first obtained by F\"{u}rstenberg 
and Katznelson~\cite{FK91} by methods from ergodic theory, asserts that the 
corresponding density statement holds as well. 
That is, for every integer $k\ge 2$ and every real $\delta>0$ there exists a dimension 
$n=\mathrm{DHJ}(k, \delta)$ such that every set $A\subseteq [k]^n$ of size 
$|A|\ge \delta k^n$ contains a combinatorial line. Nowadays some elementary 
combinatorial proofs of this so-called {\it density Hales-Jewett theorem} are known, 
see~\cites{PM12, DKT14}.
  
Our next result relates to the Hales Jewett-theorem in a similar way as 
Theorem~\ref{thm:main1} relates to van der Waerden's theorem. In order to render 
its second clause on the existence of dense subsets without lines sufficiently 
strong for our intended application, we will work with the following relaxed 
line concept.  

\begin{dfn}\label{dfn:quasi}
	Let $L$ be a $k$-element subset of the Hales-Jewett cube $[k]^n$. 
	If for every coordinate direction $\nu\in [n]$ the $\nu^{\mathrm{th}}$ 
	entries of the points in $L$ are either identical or mutually 
	distinct, then $L$ is called a {\it quasiline}. 
\end{dfn}

In general, every combinatorial line is a quasiline, but not the other way around. 
For instance, $\{124, 223, 322, 421\}$ is a quasiline in $[4]^3$, but not a 
combinatorial line. In the special case $k=3$ it may be observed that if one 
identifies $[3]$ with the three-element field and views $[3]^n$ as a vector space
over that field, then quasilines are the same as one-dimensional affine subspaces
(or arithmetic progressions of length three). However, this does not generalise
to larger prime numbers. 

\begin{theorem}\label{thm:main3}
	For all integers $k\ge 3$, $r\geq 1$ and all reals $\mu\in(0, \frac{k-1}{k})$ 
	there exist a dimension $n$ and a set of points 
	$\cX=\cX(k,r,\mu)\subseteq [k]^n$ such that
		\begin{enumerate}[label=\rmlabel]
		\item\label{it:17i} for every $r$-colouring of $\cX$ there is a monochromatic 
			combinatorial line,
		\item\label{it:17ii} but every $\cY\subseteq \cX$ contains a 
			subset $\cZ\subseteq \cY$ of size $|\cZ|\geq \mu |\cY|$ such 
			that $\cZ$ contains no quasiline.
	\end{enumerate}
\end{theorem}

In fact, the set $\cX$ we construct will also have the property 
that every quasiline $L\subseteq \cX$ is a combinatorial line. 
 
\subsection*{Organisation}
We shall show in the next section that Theorem~\ref{thm:main3} implies 
our other results, so that it will only remain to prove Theorem~\ref{thm:main3}.
The preliminary Section~\ref{sec:prelim} deals with auxiliary hypergraphs 
and restricted versions of the Hales-Jewett theorem. 
The proof of Theorem~\ref{thm:main3} itself is based on the partite construction 
method (see~\cites{NR79, FGR87}) and will be given in Section~\ref{sec:construction}. 
We conclude with some further problems and results in Section~\ref{sec:remarks}.
 
\section{Implications}
\label{sec:mainproof}

In this section we assume that Theorem~\ref{thm:main3} is true and show how 
to derive Theorem~\ref{thm:main2} from it. Since Theorem~\ref{thm:main1} agrees
with the case $d=1$ and $F=[k]$ of Theorem~\ref{thm:main2}, this means that we will 
only have to prove Theorem~\ref{thm:main3} in later sections. 

We indicate the Euclidean norm in $\RR^d$ by $\|\cdot\|$ 
and $(x, y)\longmapsto x\cdot y$ denotes the standard scalar product in the 
Euclidean space $\RR^d$.   
Here is a simple statement that will later assist us in the selection of a ``sufficiently large'' number. 

\begin{lemma}\label{lem:21}
	For every finite configuration $F\subseteq \RR^d$ there is a positive
	real $\eps=\eps(F)$ such that for all functions $\rho$, $\sigma$ 
	from $F$ to $F$ the following holds: If there is a real $q$ with 
		\begin{equation}\label{eq:1832}
		\Big| 
		q\|f'-f''\|^2
		-
		\bigl(\rho(f')-\rho(f'')\bigr)\cdot \bigl(\sigma(f')-\sigma(f'')\big)
		\Big|
		\le 
		\eps
		\quad 
		\text{for all $f', f''\in F$},
	\end{equation} 
	then there actually is a real $\overline{q}$ with 
		\begin{equation}\label{eq:1833}
		\overline{q}\|f'-f''\|^2
		=
		\bigl(\rho(f')-\rho(f'')\bigr)\cdot \bigl(\sigma(f')-\sigma(f'')\bigr)
		\quad 
		\text{for all $f', f''\in F$.}
	\end{equation}
\end{lemma}

\begin{proof}
	We shall show first that for every fixed pair $(\rho, \sigma)$ there is such 
	a constant $\eps_{\rho\sigma}$. If~\eqref{eq:1833} holds for some 
	$\overline{q}$ there is nothing to show, so we can assume that no such~$\overline{q}$
	exists. This state of affairs can be expressed in the following 
	way in the vector space $\RR^{F\times F}$. Let $v\in \RR^{F\times F}$ be the vector 
	with $(f', f'')$-entry $\|f'-f''\|^2$ for every pair $(f', f'')\in F^2$ and, 
	similarly, let $w$ be the vector with $(f', f'')$-entry 
	$\bigl(\rho(f')-\rho(f'')\bigr)\cdot \bigl(\sigma(f')-\sigma(f'')\bigr)$. 
	The absence of $\overline{q}$ means that $w$ does not belong to 
	the subspace $V=\RR v$ of $\RR^{F\times F}$ generated by~$v$. 
	Hence, there is some $\eps_{\rho\sigma}>0$ such that the distance 
	of $w$ from any point in $V$ exceeds~$|F|\eps_{\rho\sigma}$. 
	Now if~\eqref{eq:1832} held for 
	some $q\in\RR$ and for $\eps_{\rho\sigma}$ instead of $\eps$, then 
		\[
		\|qv-w\|^2
		=
		\sum_{(f', f'')\in F^2}
		\Big| 
		q\|f'-f''\|^2
		-
		\bigl(\rho(f')-\rho(f'')\bigr)\cdot \bigl(\sigma(f')-\sigma(f'')\bigr)
		\Big|^2
		\le 
		|F|^2\eps_{\rho\sigma}^2
	\]
		would contradict the choice of $\eps_{\rho\sigma}$. This concludes the proof 
	that for every pair of functions~$(\rho, \sigma)$ there is an appropriate constant 
	$\eps_{\rho\sigma}$. 
	Since there are only finitely many such pairs $(\rho, \sigma)$, the number 
	$\eps=\min\limits_{\rho\sigma}\eps_{\rho\sigma}$ is as desired. 
\end{proof}

We proceed with a finitary version of Theorem \ref{thm:main2}.

\begin{proposition}\label{thm:finiteint}
	Given a finite configuration $F\subseteq \ZZ^d$ of $k=|F|\ge 3$ points, 
	a number of colours $r\ge 1$, and a real $\mu\in(0, \frac{k-1}k)$ there 
	exists a finite set $X=X(F, r, \mu)\subseteq \NN^d$ such that
	\begin{enumerate}[label=\rmlabel]
    \item\label{it:2i} for every $r$-colouring of $X$ there is a monochromatic 
    		homothetic copy of $F$
    \item\label{it:2ii} and every $Y\subseteq X$ has an $F$-free subset $Z\subseteq Y$ 	
    		of size $|Z|\geq \mu|Y|$.
\end{enumerate}
\end{proposition}

\begin{proof}
	By translating $F$ we may assume that $F\subseteq \NN^d$.
	Let $F=\{f_1, \dots, f_k\}$ enumerate the points of $F$. Owing to 
	Theorem~\ref{thm:main3} there are a natural number $n$ and a set $\cX\subseteq [k]^n$
	such that 
		\begin{enumerate}[label=\alabel]
    \item\label{it:2a} for every $r$-colouring of $\cX$ there is a monochromatic 
    		combinatorial line
    \item\label{it:2b} and every $\cY\subseteq \cX$ possesses a 
    		subset $\cZ\subseteq \cY$ of size $|\cZ|\geq \mu |\cY|$ not containing 
			any quasilines.
	\end{enumerate}
		
	Let $\eps=\eps(F)>0$ be the constant obtained in Lemma~\ref{lem:21},
	set $s=\max\{\|f_i\|\colon i\in [k]\}$, 
	choose $T\gg n, s, \eps^{-1}$ sufficiently large, and consider the map 
		\begin{align*}
		\phi\colon \cX &\longrightarrow \ZZ^d \\
		(a(1), \dots, a(n)) & \longmapsto \sum_{i=1}^n T^{2^i}f_{a(i)}\,.
	\end{align*}
		
	Because of $T\gg s$ this map is injective. 
	We shall show that the image of $\phi$, i.e., the set $X=\phi[\cX]$, 
	has the desired properties. 
	
	Beginning with~\ref{it:2i} we look at an
	arbitrary $r$-colouring $\gamma\colon X\lra [r]$. We need to exhibit a 
	monochromatic homothetic copy of $F$.
	By~\ref{it:2a} applied to the $r$-colouring $\gamma\circ\phi$ of $\cX$ 
	there is a combinatorial embedding $\eta\colon [k]\lra \cX$ such that 
	$\gamma\circ\phi\circ \eta$ is a constant function from~$[k]$ to~$[r]$.
	The image of $\phi\circ \eta$ is clearly monochromatic and one confirms 
	easily that it is a homothetic copy of $F$.
	
	The proof of part~\ref{it:2ii} hinges on the fact that all copies of $F$ of 
	the kind we want to exclude correspond to quasilines in the Hales-Jewett cube.
	 
	\begin{claim}\label{clm:2quasi}
		Let $L\subseteq \cX$ be a set of $k$ points. If $\phi[L]$ is congruent 
		to $\lambda F$ for some nonzero real scaling factor $\lambda$, then $L$ is 
		a quasiline.
	\end{claim}
	
	In the special case relevant for Theorem~\ref{thm:main1} this has a fairly 
	simple reason briefly sketched in Remark~\ref{rem:2} below. In the general 
	case we argue as follows. 
	
	\begin{proof}[Proof of Claim~\ref{clm:2quasi}]
		Enumerate $L=\{a_1, \dots, a_k\}$ in such a way that the points $\phi(a_i)$
		in $\ZZ^d$ satisfy 
				\[
			\|\phi(a_i)-\phi(a_j)\|=\lambda \|f_i-f_j\|
		\]
				for all $i, j\in [k]$.  
		Writing $a_i=(a_i(1), \dots, a_i(n))$ for every $i\in [k]$ we contend that 
				\begin{equation}\tag{$\star$}\label{eq:PLG4}
  			\parbox{\dimexpr\linewidth-6em}{    \strut
    \it
    for all $\nu, \nu'\in [n]$ there is a real $q_{\nu,\nu'}$ such that 
			\[
				(f_{a_i(\nu)}-f_{a_j(\nu)})\cdot (f_{a_i(\nu')}-f_{a_j(\nu')})
				=
				q_{\nu,\nu'}\|f_i-f_j\|^2
			\]
						holds for all $i, j\in [k]$.
    \strut
  }
\end{equation}

		Assume for the sake of contradiction that this fails and fix a 
		counterexample $(\nu, \nu')$ for which $2^\nu+2^{\nu'}$ is maximal. 
		It is important to note here that this condition determines the 
		pair $\{\nu, \nu'\}$ uniquely, because every integer can be written 
		in at most one way as a sum of two (identical or distinct) powers of two. 
		Setting 
		$D=\{(\mu, \mu')\in [n]^2\colon 2^\mu+2^{\mu'}>2^\nu+2^{\nu'}\}$ 
		our extremal choice of $(\nu, \nu')$ ensures that for every pair 
		$(\mu, \mu')\in D$ there exists an appropriate constant $q_{\mu,\mu'}$.
		Now for all $i, j\in [k]$ we have 
				\begin{align*}
			\lambda^2 \|f_i-f_j\|^2
			&=
			\|\phi(a_i)-\phi(a_j)\|^2
			=
			\bigg\| \sum_{\mu\in [n]} T^{2^\mu}(f_{a_i(\mu)}-f_{a_j(\mu)})\bigg\|^2 \\
			&=
			\sum_{(\mu, \mu')\in D} T^{2^\mu+2^{\mu'}}
				(f_{a_i(\mu)}-f_{a_j(\mu)})\cdot (f_{a_i(\mu')}-f_{a_j(\mu')})\\
			&\phantom{\sum_{(\mu, \mu')\in D}}+(2-\delta_{\nu,\nu'}) T^{2^\nu+2^{\nu'}}
				(f_{a_i(\nu)}-f_{a_j(\nu)})\cdot(f_{a_i(\nu')}-f_{a_j(\nu')})
			+O(T^{2^\nu+2^{\nu'}-1})\,,
		\end{align*}
		where $\delta$ denotes Kronecker's delta and the implied constant depends 
		only on $n$ and $s$. Simplifying the sum over $D$ on the right side with 
		the help of~\eqref{eq:PLG4} we see that the number
		\[
			q
			=
			\frac{\lambda^2-\sum_{(\mu, \mu')\in D}T^{2^\mu+2^{\mu'}}q_{\mu,\mu'}}
			{(2-\delta_{\nu,\nu'}) T^{2^\nu+2^{\nu'}}}
		\]
				satisfies
				\[
			q\|f_i-f_j\|^2
			=
			(f_{a_i(\nu)}-f_{a_j(\nu)})\cdot (f_{a_i(\nu')}-f_{a_j(\nu')})+O(T^{-1})
		\]
				for all $i, j\in [k]$. In terms of the functions $\rho$ and $\sigma$
		from $F$ to $F$ defined by $\rho(f_i)=f_{a_i(\nu)}$ 
		and $\sigma(f_i)=f_{a_i(\nu')}$ for all $i\in [k]$ this means
				\[
			q\|f_i-f_j\|^2
			=
			\bigl(\rho(f_i)-\rho(f_j)\bigr)\cdot \bigl(\sigma(f_i)-\sigma(f_j)\bigr)
			+O(T^{-1})\,.
		\]
				But due to $T\gg n, s, \eps^{-1}$ and our choice 
		of $\eps$ this implies that there is a constant $q_{\nu,\nu'}$ 
		such that 
		\[
			q_{\nu,\nu'}\|f_i-f_j\|^2
			=
			\bigl(\rho(f_i)-\rho(f_j)\bigr)\cdot \bigl(\sigma(f_i)-\sigma(f_j)\bigr)
		\]
		holds for all $i, j\in [k]$. That is, $q_{\nu,\nu'}$ has the property 
		demanded by~\eqref{eq:PLG4} and, thereby, the proof of~\eqref{eq:PLG4}
		is complete. 
		
		In the special case $\nu'=\nu$ we obtain 
		\[
			q_{\nu,\nu}\|f_i-f_j\|^2
			=
			\|f_{a_i(\nu)}-f_{a_j(\nu)}\|^2
		\]
				for all $i, j\in [k]$.
		Hence, for every fixed $\nu$ the map $f_i\longmapsto f_{a_i(\nu)}$
		sends $F$ to a congruent copy of $\sqrt{q_{\nu,\nu}}F$. In the special 
		case $q_{\nu,\nu}=0$ this means $a_1(\nu)=\dots=a_k{(\nu)}$ and 
		if $q_{\nu,\nu}\ne 0$ we have, in particular, $\{a_i(\nu)\colon i\in [k]\}=[k]$.
		For these reasons, $L$ is indeed a quasiline. 
	\end{proof}

	After this preparation, part~\ref{it:2ii} of the theorem is straightforward. 
	Let an arbitrary set $Y\subseteq X$ be given. Due to~\ref{it:2b} the set 
	$\cY=\phi^{-1}[Y]\subseteq \cX$ has a subset $\cZ\subseteq \cY$ 
	of size $|\cZ|\geq \mu |\cY|$ containing no quasilines. By Claim~\ref{clm:2quasi}
	the set $Z=\phi(\cZ)$ is $F$-free and, since $\phi$ is injective, 
	it is also sufficiently dense. 
\end{proof}

\begin{remark}\label{rem:2}
	Here is a simpler proof for the special case $d=1$ and $F=[k]$ 
	of Claim~\ref{clm:2quasi}. Again we write $L=\{a_1, \dots, a_k\}$
	and $a_i=(a_i(1), \dots, a_i(n))$ for every $i\in [k]$. 
	Since $\phi(a_1), \dots, \phi(a_k)$ is an $\AP_k$, we have 
		\[
		\sum_{\nu=1}^n T^{2^\nu}\bigl(a_{i+1}(\nu)-2a_i(\nu)+a_{i-1}(\nu)\bigr)
		=
		\phi(a_{i+1})-2\phi(a_i)+\phi(a_{i-1})
		=0
	\]
		for every $i\in [2, k-1]$. Thus a sufficiently large choice of $T$ guarantees 
	that for every $\nu\in [n]$ the $k$-tuple 
	$A_\nu=(a_1(\nu), \dots, a_k(\nu))\in [k]^k$
	is a (possibly degenerate) arithmetic progression of length $k$. 
	So $A_\nu$ either consists of $k$ equal numbers, or it is one of 
	the two $k$-tuples $(1, 2, \dots, k)$ or $(k, k-1, \dots, 1)$.
	In particular, $L$ is indeed a quasiline.  
\end{remark}

It remains to deduce Theorem~\ref{thm:main2} from the finitary version we have 
just obtained. 

\begin{proof}[Proof of Theorem \ref{thm:main2}]
	For every $r\ge 1$ let $X_r=X(F, r, \mu)$ be the set generated by 
	Proposition~\ref{thm:finiteint}. Take a sequence $(v_r)_{r\ge 1}$ of vectors 
	in $\NN^d$ such that $\|v_r\|$ tends to infinity sufficiently fast and 
	define
		\[
		X=\bigdcup_{r\ge 1}(v_r+X_r)\,.
	\]
		Provided that this set satisfies the conclusion of the following claim, we shall 
	show later that it has the properties demanded by Theorem~\ref{thm:main2}. 
	
	\begin{claim}\label{clm:vrfast}
		An appropriate choice of $(v_r)_{r\ge 1}$ ensures that 
		if a configuration $F'\subseteq X$ is congruent to $\lambda F$ 
		for some nonzero real $\lambda$, then $F'\subseteq v_r+X_r$ holds 
		for some $r\ge 1$.
	\end{claim}
	
	\begin{proof}
		Let $r$ be maximal such that $F'\cap (v_r+X_r)\ne\vn$. The main 
		point is that when choosing $v_r$ the set $X_{<r}=\bigdcup_{s<r}(v_s+X_s)$
		has already been determined. Moreover, the maximality of $r$ 
		yields $F'\subseteq X_{<r}\dcup (v_r+X_r)$. Because of $|F'|=k\ge 3$
		this means that at least one of the sets $X_{<r}$ and $v_r+X_r$ 
		needs to contain at least two points of $F'$. 
		
		As we can force $v_r+X_r$ to be as far apart from $X_{<r}$ as we want, 
		we can thus guarantee that $F'$ is a subset of either $X_{<r}$ or $v_r+X_r$.
		Together with $F'\cap (v_r+X_r)\ne\vn$ this implies $F'\subseteq v_r+X_r$.
	\end{proof}
	
	Since the set $X$ contains for every $r\ge 1$ a translated copy of $X_r$, 
	it has the first property promised by Theorem~\ref{thm:main2}.
	In order to establish the second property, we consider an arbitrary finite
	set $Y\subseteq X$ and set $Y_r=Y\cap (v_r+X_r)$ for every $r\ge 1$. Since  
	part~\ref{it:2ii} of Proposition~\ref{thm:finiteint} is translation-invariant, 
	there are $F$-free subsets $Z_r\subseteq Y_r$ of size $|Z_r|\ge \mu|Y_r|$.
	The subset $Z=\bigcup_{r\ge 1}Z_r$ of $Y$ clearly satisfies $|Z|\ge \mu |Y|$
	and Claim~\ref{clm:vrfast} implies that it is $F$-free as well.   
\end{proof}

\section{Preliminaries}\label{sec:prelim}
\subsection{The {\texorpdfstring{$\mu$}{mu}}-fractional property} 
The combinatorial lines in the set $\cX(k, r, \mu)$ we need to construct will 
certainly form a hypergraph $H$ with the special property that every subset 
of~$V(H)$ contains a large independent set (consisting of a $\mu$-proportion of 
its elements). Later it turns out to be helpful to work with a weighted version 
of this property.   

\begin{dfn}
	A $k$-uniform hypergraph $H$ has the {\it $\mu$-fractional property}
	for a real $\mu\in (0, 1]$ 
	if for every family $(w_i)_{i\in V(H)}$ of nonnegative real numbers 
	there exists an independent set $Z\subseteq V(H)$ such that 
	$\sum_{i\in Z}w_i\ge \mu \sum_{i\in V(H)}w_i$.
\end{dfn}

Let us observe that if a hypergraph $H$ has this property, then for every 
set $Y\subseteq V(H)$ we obtain an independent subset $Z\subseteq Y$ of size 
$|Z|\ge \mu |Y|$ by considering the characteristic function of $Y$. 
When we want to check whether a given hypergraph $H$ has the $\mu$-fractional property,
we can always assume that the given family $(w_i)_{i\in V(H)}$ satisfies 
$\sum_{i\in V(H)}w_i=1$. This is because the case that this sum vanishes 
is trivial, and otherwise we can divide all weights~$w_i$ by their sum without changing 
the situation. 

The advantage of allowing arbitrary weights $w_i\ge 0$ as opposed to just 
working with $w_i\in\{0, 1\}$ is that 
thereby the property is not only preserved under taking subhypergraphs, but also 
under taking blow-ups. We express this fact as follows. 

\begin{lemma}\label{cl:hereditary}
	Suppose $\mu\in (0, 1]$ and that $G$, $H$ are two $k$-uniform hypergraphs 
	for which there exists a homomorphism $\psi$ from $G$ to $H$. 
	If $H$ has the $\mu$-fractional property, then so does $G$. 
	In particular, the $\mu$-fractional property is hereditary, i.e., if some hypergraph
	has the property, then so do all of its subhypergraphs. 
\end{lemma}

\begin{proof}
	Let a family $(w_i)_{i\in V(G)}$ of nonnegative real numbers summing up to $1$
	be given and set $u_j=\sum_{i\in \psi^{-1}(j)}w_i$ for every $j\in V(H)$.
	Since $\sum_{j\in V(H)}u_j=1$ and $H$ has the $\mu$-fractional property,
	there is an independent set $Z_H\subseteq V(H)$ such 
	that $\sum_{j\in Z_H}u_j\ge \mu$. Now $Z_G=\psi^{-1}[Z_H]$ is 
	independent in $G$ (because $\psi$ is a homomorphism), and we 
	have $\sum_{i\in Z_G}w_i=\sum_{j\in Z_H}u_j\ge \mu$. 
	This shows that $G$ has indeed the $\mu$-fractional property. 
	
	It remains to remark that if $G$ is a subhypergraph of $H$, 
	then the inclusion map $V(G)\lra V(H)$ is a hypergraph homomorphism. 
\end{proof}

\subsection{Auxiliary hypergraphs}\label{sec:auxgraphs}

Let us recall that for positive integers $n, \ell\ge k\ge 2$ the {\it $k$-uniform
shift hypergraph $H=\mathrm{Sh}^{(k)}(n, \ell)$ on the $\ell$-subsets of $[n]$} 
is defined to have the vertex set $V(H)=[n]^{(\ell)}$ and the following 
$\binom n{k+\ell-1}$ edges: for every increasing sequence $a_1<\dots <a_{k+\ell-1}$
of integers from $[n]$ there is an edge $\{x_1,\ldots,x_k\}\in E(H)$ obtained 
by setting $x_i=\{a_i,\ldots,a_{i+\ell-1}\}$ for every $i\in [k]$. 

The key property of these shift hypergraphs we exploit in this work extends an idea 
from~\cite{EHS82}.
Roughly speaking, the result says
that if we take $\ell$ large enough, then $\mathrm{Sh}^{(k)}(n, \ell)$ has 
the $\mu$-fractional property for some $\mu$ as close to $\frac{k-1}k$ as we want. 
 
More precisely, given an integer $k\ge 2$ and a real $\mu\in (0, \frac{k-1}k)$
we first set  
\[
	\ell
	=
	\left\lceil\frac{2(k-1)^2}{(k-1)-k\mu}\right\rceil
\]
and then we consider $H^{(k)}(n, \mu)=\mathrm{Sh}^{(k)}(n, \ell)$ for every $n\ge k$.
A proof of the following result, which was suggested by Paul Erd\H{o}s, can be 
found in~\cite{NRS}*{Section 5}. For the reader's convenience we include a brief 
sketch of a simplified version of the argument below.

\begin{theorem}[Ne\v{s}et\v{r}il, R\"odl, and Sales]\label{thm:graphpisier}
	For all integers $k\ge 2$, $r\ge 1$ and every real $\mu\in(0, \frac{k-1}{k})$ 
	there exists an integer $n=n(k, r, \mu)$ such that the $k$-uniform hypergraph 
	$H=H^{(k)}(n, \mu)$ satisfies $\chi(H)>r$ and has the $\mu$-fractional property. 
\end{theorem}

\begin{proof}
	The claim on the chromatic number follows easily from Ramsey's 
	theorem~\cite{R29}. Next, our choice of $\ell$ guarantees that the set 
		\[
		I=\bigl\{i\in [k, \ell-k+1]\colon i\not\equiv -1\pmod{k}\bigr\}
	\]
		satisfies $|I|\ge \mu\ell$. Given a permutation $\pi\in\mathfrak{S}_n$ 
	and a vertex $x=\{a_1, \ldots, a_\ell\}\in V(H)$, where $a_1<\dots<a_\ell$,
	we denote the unique index $j\in [\ell]$ such that 
	$\pi(a_j)=\max\{\pi(a_i)\colon i\in [\ell]\}$ by~$\nu(x, \pi)$. 
	It is not difficult to check that for every permutation $\pi$ the set 
		\[
		Y_\pi=\{x\in V(H)\colon \nu(x, \pi)\in I\}
	\]
		is independent in $V(H)$. Moreover, if $(w_x)_{x\in V(H)}$ is a family of
	nonnegative real weights summing up to $1$ and $\pi$ gets chosen uniformly 
	at random, then the expectation of $\sum_{x\in Y_\pi}w_x$ is $|I|/\ell$.
	Hence, there exists some $\pi\in\mathfrak{S}_n$ such that 
	$\sum_{x\in Y_\pi}w_x\ge |I|/\ell\ge \mu$.
\end{proof}

In the special case $k=3$ we shall also need another property of shift hypergraphs.
Let~$K_4^{(3)-}$ denote the $3$-uniform hypergraph with four vertices and three edges. 

\begin{lemma}\label{lem:k43-}
	For all $n, \ell\ge 3$ the shift hypergraph $\mathrm{Sh}^{(3)}(n, \ell)$
	is $K_4^{(3)-}$-free. 
\end{lemma}

\begin{proof}
	Given a tournament $T$, Erd\H{o}s and Hajnal introduced the $3$-uniform 
	{\it tournament hypergraph} $H(T)$ which has the same vertices as $T$ 
	and whose edges correspond to the cyclically oriented triangles in $T$. 
	It is well known that these tournament hypergraphs are $K_4^{(3)-}$-free 
	(see, e.g.,~\cite{ErHa72}). 
	
	Thus it suffices to orient the pairs of vertices 
	of $H=\mathrm{Sh}^{(3)}(n, \ell)$ in such a way that all edges of $H$
	induce cyclically oriented triangles. Consider any such 
	pair $\{x, y\}\in V(H)^{(2)}$. 
	If $\min(x)=\min(y)$, the orientation 
	of $xy$ is immaterial (because no edge of~$H$ contains both~$x$ and~$y$). 
	If $\min(x)<\min(y)$, we choose the orientation $x\to y$ or $y\to x$
	depending on whether $|y\sm x|$ is even or odd. Now for every edge 
	$\{x, y, z\}\in E(H)$ with $\min(x)<\min(y)<\min(z)$ we have the oriented
	triangle $z\to y\to x\to z$. 
\end{proof}

\begin{remark}
	More generally it could be shown that for $n, \ell\ge k\ge 2$ 
	the $k$-uniform shift hypergraph $\mathrm{Sh}^{(k)}(n, \ell)$
	is $F^{(k)}$-free, where $F^{(k)}$ denotes the $k$-uniform 
	hypergraph on $k+1$ vertices with three edges.   
	One way to see this involves higher order tournaments 
	(described, e.g., in~\cite{3edges}*{\S1.3}), which are known 
	to be $F^{(k)}$-free (see~\cite{3edges}*{Fact~1.5}). 
\end{remark}

\subsection{Triangles and tripods in Hales--Jewett cubes}\label{sec:halesjewett}
Given an arbitrary finite set $A$, one can form Hales-Jewett cubes $A^n$ 
and define combinatorial lines as in~\S\ref{subsec:cl}. In this context 
one often calls $A$ the `alphabet' and the points in $A^n$ are then viewed as 
`words of length~$n$'. We shall write $\ccL(A^n)$ for the collection of all 
combinatorial lines in $A^n$. For simplicity we identify 
any subset $\ccL\subseteq \ccL(A^n)$ with the $|A|$-uniform hypergraph on~$A^n$
whose set of edges is $\ccL$. We may thus write $\chi(\ccL)$ for the chromatic 
number of this hypergraph. With this notation the Hales-Jewett theorem states 
that for every fixed alphabet $A$ we have $\lim_{n\to\infty}\chi(\ccL(A^n))=\infty$.

In our construction we need the existence of certain `sparse' subhypergraphs 
$\ccL\subseteq \ccL(A^n)$ of large chromatic number. Let us note first that 
$\ccL(A^n)$ itself is linear, i.e., any two of its edges intersect in at most 
one vertex. This follows from the obvious fact that through any two distinct 
points of a Hales-Jewett cube there can pass at most one combinatorial line. 
Three distinct lines in $\ccL(A^n)$ are said to form a {\it triangle} if 
they do not pass through a common point, but any two of them intersect. 

As proved by the second author~\cite{R90}, given~$A$ and~$r$ there is 
for some dimension $n$ a triangle-free line system $\ccL\subseteq\ccL(A^n)$
such that $\chi(\ccL)>r$. In fact, he even showed that hypergraphs of large 
chromatic number and large girth, first obtained by Erd\H{o}s, Hajnal, 
and Lov\'{a}sz using different methods~\cites{E59,EH66, L68}, can be found 
inside the Hales-Jewett hypergraphs~$\ccL(A^n)$. For the present purposes 
excluding triangles is important, but longer cycles are irrelevant. 

There is, however, one further configuration of lines that we need to forbid. 
In the definition that follows, for every combinatorial line $L\subseteq A^n$
its set of moving coordinates is denoted by $M_L$.

\begin{dfn}
	Three distinct combinatorial lines $L, L', L''\subseteq A^n$ passing through 
	a common point are said to form a {\it tripod} if $M_L$ is the disjoint union 
	of $M_{L'}$ and $M_{L''}$.
\end{dfn}

For instance, for every $a\in A$ the diagonal $\{xx\colon x\in A\}$ forms 
together with the two lines $\{ax\colon x\in A\}$ and $\{xa\colon x\in A\}$
a tripod in $A^2$. It turns out that the argument in~\cite{R90} allows to 
exclude tripods and short cycles at the same time. We will only state and 
prove the case of triangles here. 

\begin{thm}\label{thm:tripod}
	Given an alphabet $A$ with at least two letters and $r\ge 1$, there is for 
	every sufficiently large dimension $n$ a collection $\ccL\subseteq \ccL(A^n)$ 
	of combinatorial lines containing neither tripods nor triangles such 
	that $\chi(\ccL)>r$.
\end{thm}

\begin{proof}
	For transparency we can assume $A=[k]$, where $k\ge 2$.
	Depending on $k$ and $r$ we fix a real $\alpha>0$ and natural numbers 
	$d$, $m$, $n$ fitting into the hierarchy 
		\[
		n\gg d\gg \alpha^{-1} \gg m \gg k, r\,.
	\]
		
	For every $i\in [m]$ let $\ccL_i$ be the collection of all combinatorial lines 
	$L\subseteq [k]^n$ with $|M_L|=i$. Since there are $\binom ni$ possibilities 
	for the set $M_L$ and $k^{n-i}$ possibilities for the behaviour of $L$ on the 
	constant coordinates, we have $|\ccL_i|=\binom ni k^{n-i}$.
	
	\begin{claim}\label{clm:3a}
		For every colouring $\gamma\colon [k]^n\lra [r]$ there is some $i\in [m]$
		such that at least~$\alpha |\ccL_i|$ of the lines in $\ccL_i$ are monochromatic 
		with respect to $\gamma$.   
	\end{claim} 
	
	\begin{proof}
		We can regard $[k]^n$ as the set of all function from $[n]$ to $[k]$. 
		Let $\Omega$ be the set of all~$\binom nm k^{n-m}$ functions from an
		$(n-m)$-element subset of~$[n]$ to~$[k]$. For each of these functions $f$
		the set $S_f=\{g\in [k]^n\colon g\supseteq f\}$ of all points
		extending it is an isomorphic copy 
		of the Hales-Jewett cube $[k]^m$. So by the Hales-Jewett theorem there 
		is some monochromatic combinatorial line $L_f\subseteq S_f$. This line
		belongs to one of the sets $\ccL_1, \dots, \ccL_m$ and the box principle 
		(Schubfachprinzip) yields some set $\Omega'\subseteq\Omega$ of 
		size $|\Omega'|\ge \frac 1m |\Omega|$
		together with an integer $i\in [m]$ such that $L_f\in \ccL_i$ holds for every 
		$f\in \Omega'$. Conversely, every line $L\in \ccL_i$ appears in~$\binom{n-i}{m-i}$
		of the spaces $S_f$ with $f\in \Omega$. For these reasons, the number of 
		monochromatic lines in $\ccL_i$ is at least 
				\[
			\frac{|\Omega'|}{\binom{n-i}{m-i}}
			\ge
			\frac 1m \left(\frac nm\right)^i k^{n-m}
			\ge
			\frac{k^{i-m}}{m^{i+1}}|\ccL_i|
			\ge
			\alpha |\ccL_i|\,. \qedhere
		\]
	\end{proof}
	
	Let us call a collection of lines $\ccL\subseteq \bigdcup_{i\in [m]}\ccL_i$
	{\it suitable} if 
		\begin{enumerate}[label=\nlabel]
		\item\label{it:s1} through every point $x\in [k]^n$ there pass at most $d$ 
			lines from $\ccL$;
		\item\label{it:s2} no three lines in $\ccL$ form a tripod or a triangle.
			\end{enumerate}
		
	For instance, $\vn$ is a suitable collection of lines. The idea for constructing 
	the desired system of lines is that starting with $\vn$ we keep adding lines one 
	by one while maintaining at every step that the set of lines we have already chosen 
	remains suitable. It can be shown that as long as we have selected at most 
		\[
		q=\frac{2m\log(r)}{\alpha}k^n
	\]
		lines, we can still choose `almost every' line in the next step. This will in turn 
	imply that in every step we can reduce the number of `bad' colourings, 
	which have no monochromatic line in our system yet, by a constant proportion.
	At most $q$ such steps will push the number of bad colourings below one. 
	
	\begin{claim}\label{clm:3b}
		If $\ccL$ is a suitable system of at most $q$ lines, then for every $i\in[m]$
		all but at most $\alpha |\ccL_i|/2$ lines $L\in \ccL_i$ have the property that 
		$\ccL\cup \{L\}$ is again suitable. 
	\end{claim}
	
	\begin{proof}
		Fix $i\in [m]$.
		We shall first bound the number $s_1$ of lines in $\ccL_i$ whose addition 
		to $\ccL$ would cause a violation of~\ref{it:s1}. Let $A\subseteq [k]^n$ be 
		the set of all points lying on exactly $d$ lines from $\ccL$. Since every 
		line contains $k$ points, double counting yields $|A|d\le k|\ccL|\le kq$.
		Together with the fact that through every point of $[k]^n$ there pass 
		at most $\binom ni$ lines from~$\ccL_i$ this shows
				\begin{equation}\label{eq:s1}
			s_1
			\le 
			\binom ni |A|
			\le 
			\frac{kq}{dk^{n-i}}|\ccL_i|
			=
			\frac{2k^{i+1}m\log(r)}{d\alpha}|\ccL_i|
			\le 
		   \frac{\alpha}{4}|\ccL_i|\,.
		 \end{equation}
		 		 
		 Next, the number $s_2$ of lines whose addition to $\ccL$ would create a tripod 
		 can be bounded by 
		 		 \begin{equation}\label{eq:s2}
		 	s_2\le d^2k^n\,.
		 \end{equation}
		 		 This is because there are $k^n$ possibilities for 
		 a point $x\in [k]^n$, where the three lines of such a tripod could meet,
		 and due to~\ref{it:s1} there are at most $d^2$ pairs of 
		 lines $\{L', L''\}\in \ccL^{(2)}$ passing through $x$. 
		 Moreover, given $L'$ and $L''$ there is at most one line $L$ 
		 completing a tripod. 
		 
		 Utilising that through any two points there is at most one line one 
		 shows similarly that the number $s_3$ of lines $L$ for 
		 which $\ccL\cup\{L\}$ contains a triangle can be bounded by 
		 		 \[
		 	s_3\le (kd)^2 k^n\,.
		 \]
		 		 Together with~\eqref{eq:s2} this shows 
		 		 \[
		 	s_2+s_3
			\le
			2d^2k^{n+2}
			=
			\frac{2k^{i+2}d^2}{\binom ni}|\ccL_i|
			\le
			\frac{2k^3d^2}{n}|\ccL_i|
			\le
			\frac\alpha 4|\ccL_i|\,.
		\]
				In view of~\eqref{eq:s1} the desired estimate $s_1+s_2+s_3\le \alpha|\ccL_i|/2$
		follows.     
	\end{proof}
	
	Now for every system of lines $\ccL\subseteq\ccL([k]^n)$ we denote the set 
	of all `bad' colourings $\gamma\colon [k]^n\lra [r]$ such that no line in $\ccL$
	is monochromatic with respect to $\gamma$ by $B(\ccL)$. 
	Take a maximal suitable line system $\ccL$ with the 
	property  
		\[
		 |B(\ccL)|\le (1-\alpha/2m)^{|\ccL|}r^{k^n}\,.
	\]
	The existence of such a system is guaranteed by the fact 
	that $|B(\vn)|\le r^{k^n}$. If $|\ccL|>q$, then 
		\[
		|B(\ccL)|<\exp(-q\alpha/2m+k^n\log(r))=1
	\]
		proves that all colourings are good for $\ccL$, which in turn means that $\ccL$
	has all properties promised by the theorem. So we can assume $|\ccL|\le q$
	in the sequel. By Claim~\ref{clm:3a} and the box principle there are a set 
	$B\subseteq B(\ccL)$ of size $|B|\ge \frac1m |B(\ccL)|$ and an integer $i\in [m]$
	such that for every colouring in $B$ at least $\alpha |\ccL_i|$ lines in $\ccL_i$
	are monochromatic. Now Claim~\ref{clm:3b} reveals that for every colouring 
	in $B$ there are at least $\alpha |\ccL_i|/2$ monochromatic lines $L\in \ccL_i$ 
	for which $\ccL\cup \{L\}$ is suitable. Consequently, there is a fixed 
	line $L\in \ccL_i$ which is monochromatic for at least $\alpha |B|/2$ colourings 
	in~$B$ such that $\ccL^\star=\ccL\cup \{L\}$ is suitable. 
	But now 
		\[
		|B(\ccL^\star)|
		\le 
		|B(\ccL)|-\alpha |B|/2 
		\le 
		(1-\alpha/2m)|B(\ccL)|
		\le 
		(1-\alpha/2m)^{|\ccL^\star|}r^{k^n} 
	\]
		shows that $\ccL^\star$ contradicts the maximality of $\ccL$.
\end{proof}

\begin{remark}
	If we just wanted to produce a line system of large chromatic number 
	without triangles (or short cycles), we could also use the partite construction 
	method. However, one of us is bamboozled by the fact that he cannot exclude 
	tripods in this way.  
\end{remark}

\subsection{More on embeddings.}
Preparing a concise description of the partite construction we shall perform 
in the next section, we would like to offer some (mostly standard) 
remarks on combinatorial embeddings. For a fixed alphabet $A$ and natural
numbers $n\ge m$ a map $\eta\colon A^m\lra A^n$ is called a {\it combinatorial 
embedding} if there are a partition $[n]=C\dcup M_1\dcup \dots \dcup M_m$
and a function $g\colon C\lra A$ such that $M_1, \dots, M_m\ne\vn$ and 
for every $a=(a_1, \dots, a_m)\in A^m$ and every $i\in [n]$ the $i^{\mathrm{th}}$ 
coordinate of $\eta(a)$ is 
\[
	\begin{cases}
		g(i) & \text{ if $i\in C$,} \cr
		a_j  & \text{ if $i\in M_j$.} \cr 
	\end{cases}
\]
In the special case $m=1$ this reduces to the notion of combinatorial 
embeddings ${A\lra A^n}$ introduced in~\S\ref{subsec:cl}. It is well known 
and easy to verify that every composition of combinatorial embeddings 
$A^\ell\lra A^m\lra A^n$ is again a combinatorial embedding. This implies,
for instance, that combinatorial embeddings map combinatorial lines to 
combinatorial lines. Similarly, quasilines are mapped to quasilines.

For $|A|\ge 2$ the partition $[n]=C\dcup M_1\dcup \dots \dcup M_m$
and the function $g\colon C\lra A$ are uniquely determined by the 
corresponding embedding $\eta$. Thus for every superset $B\supseteq A$
there is a unique extension of $\eta$ to a combinatorial embedding 
$\eta^+\colon B^m\lra B^n$.

We will only encounter such extensions in the following context.
For some set $\Pi_x\subseteq [k]^m$ we have a combinatorial embedding 
$\eta\colon \Pi_x\lra \Pi_x^n$. Identifying $([k]^m)^n$ in the obvious 
manner with $[k]^{mn}$ we then get an extension $\eta^+\colon [k]^m\lra [k]^{mn}$.
By construction, $\eta^+$ is a combinatorial embedding from the one-dimensional 
space over $[k]^m$ to the $n$-dimensional space over $[k]^m$. It is readily verified  
that we can also view $\eta^+$ over the smaller alphabet~$[k]$ as a combinatorial 
embedding from $m$-dimensional space into $(mn)$-dimensional space. 
Consequently, and this is something we shall exploit later, compositions of such 
extensions are combinatorial embeddings over $[k]$ as well.

\section{The partite construction}\label{sec:construction}

The proof of Theorem~\ref{thm:main3} is based on the partite 
construction method (see \cites{NR79, FGR87}). This means that 
we will recursively construct a sequence of ``pictures'' 
$\Pi_0, \ldots, \Pi_q$, the last one of which corresponds to the 
desired set $\cX(k, r, \mu)$. The entire construction will take 
place ``over'' a hypergraph $G$ obtained in 
Theorem~\ref{thm:graphpisier}. The pictures $\Pi_i$ themselves 
will consist of subsets $P_i$ of some Hales-Jewett cubes 
$[k]^{m_i}$ and maps $\psi_i\colon P_i\lra V(G)$ telling 
us in which way the points $x\in P_i$ are associated to 
vertices $\psi_i(x)$ of $G$.

Throughout the construction, we need to pay attention to the 
combinatorial lines in these sets $P_i$. In picture zero~$\Pi_0$ 
they are mutually disjoint and there will be one line for every 
edge of $G$. While constructing $\Pi_0, \dots, \Pi_q$ one of our 
aims is to transfer the property $\chi(G)>r$ of $G$ gradually onto 
the pictures in our sequence. Moreover, clause~\ref{it:17ii} of 
Theorem~\ref{thm:main3} forces us to protect ourselves as much 
as possible against unwanted quasilines in our pictures. 
In general, partite constructions (when executed carefully) 
tend to produce Ramsey objects that are locally quite sparse and we will benefit 
from this phenomenon as well. 

\subsection{Pictures}
In the context of the present work, pictures are defined as follows. 

\begin{dfn}\label{dfn:pic}
	Let $G$ be a $k$-uniform hypergraph, where $k\ge 3$. 
	A {\it picture} over $G$ is a pair $\Pi=(P, \psi)$
	consisting of a subset $P\subseteq [k]^m$ of a 
	Hales-Jewett cube and a map $\psi\colon P\lra V(G)$ 
	such that every quasiline $L\subseteq P$ is a 
	combinatorial line satisfying $\psi[L]\in E(G)$.
\end{dfn}

If $\Pi=(P, \psi)$ is a picture over $G$ and $x$ is a vertex of $G$, 
the set $\Pi_x=\psi^{-1}(x)$ is called the {\it music line} over $x$. 
Clearly, $P$ is the disjoint union of all music lines. In our figures 
we will always draw the hypergraph $G$ vertically to the left side 
of $P$, and $P$ itself will be drawn in such a way that every vertex 
$x\in V(G)$ is together with its music line $\Pi_x$ on a common horizontal 
line. Thus $\psi$ can be thought of as a projection to the left side 
(see, e.g., Figure~\ref{fig:fig1}). 
We prepare the construction of picture zero by showing that there are 
arbitrarily many lines ``in general position''.

\begin{lemma}\label{lem:41}
	For all integers $k\ge 3$ and $m\ge 1$ there are mutually disjoint combinatorial 
	lines $L_1, \dots, L_m\subseteq [k]^{2m}$ such that the only quasilines 
	$L\subseteq \bigdcup_{i\in [m]}L_i$ are $L_1, \dots, L_m$ themselves. 
\end{lemma}

\begin{proof}
	For every $i\in [m]$ we define $L_i$ to be a line whose only moving coordinate 
	is $i$. We further require that all points of $L_i$ have the entry $2$ in the 
	$(m+i)^\mathrm{th}$ coordinate and the entry $1$ in all other constant coordinates. 
	So, e.g., if $m=3$, we take the three lines $L_1=\{x11211\colon x\in [k]\}$, 
	$L_2=\{1x1121\colon x\in [k]\}$, and $L_3=\{11x112\colon x\in [k]\}$. 
	
	It is plain that these $m$ lines are mutually disjoint. 
	Now let $L\subseteq L_1\dcup\dots\dcup L_m$ be a quasiline. Clearly 
	there is some $j\in [m]$ such that $L\cap L_j\ne \vn$ and it suffices 
	to show that~$L=L_j$. To this end we observe that for every $i\in [m]$
	the points of $L$ can only have the entries $1$ or $2$ in their $(m+i)^\mathrm{th}$ 
	coordinates. Therefore, all points of $L$ need to agree in these coordinates 
	(cf.~Definition~\ref{dfn:quasi}) and together with $L\cap L_j\ne \vn$ it follows 
	that the $m$ last coordinates of the points in $L$ and $L_j$ are the same. 
	Combined with $L\subseteq \bigdcup_{i\in [m]}L_i$ this leads to $L=L_j$. 
\end{proof}

\begin{lemma}[Picture zero]\label{lem:Pic0}
	If $k\ge 3$ and $G$ denotes a $k$-uniform hypergraph, then there is a
	picture $\Pi_0=(P_0, \psi_0)$ over $G$ such that there is a 
	family $(L_e)_{e\in E(G)}$ of mutually disjoint combinatorial 
	lines satisfying $P_0=\bigdcup_{e\in E(G)}L_e$ and $\psi_0[L_e]=e$
	for every $e\in E(G)$.
\end{lemma}

\begin{proof}
	Set $m=|E(G)|$, fix an arbitrary enumeration $E(G)=\{e_1, \dots, e_m\}$,
	and consider the combinatorial lines ${L_1, \dots, L_m\subseteq [k]^{2m}}$   
	obtained in Lemma~\ref{lem:41}. Define $L_{e_i}=L_i$ for every $i\in [m]$
	and set $P_0=\bigdcup_{e\in E(G)}L_e$. Since these lines are mutually disjoint, there 
	is a map $\psi_0\colon P_0\lra V(G)$ such that $\psi_0[L_e]=e$ holds for 
	every $e\in E(G)$. Now $(P_0, \psi_0)$ is the desired picture. 
\end{proof}

Graphically, the picture $\Pi_0$ can be represented as in Figure \ref{fig:fig1}. 
On the vertical projection we have our $k$-uniform hypergraph $G$ with labelled edges 
$\{e_1,\ldots,e_m\}$. For each edge $e_j$ there is a corresponding combinatorial 
line $L_j$ drawn in the same colour. The music lines $\Pi_{0,x}=\psi_0^{-1}(x)$ 
are visualised as dashed horizontal lines. 

\begin{figure}[h]
\centering
{\hfil \begin{tikzpicture}[scale=0.7]    
    \coordinate (p0) at (0,8);
    \coordinate (p1) at (0,7);
    \coordinate (p2) at (0,6);
    \coordinate (p3) at (0,5);
    \coordinate (p4) at (0,4);
    \coordinate (p5) at (0,3);
    \coordinate (p6) at (0,2);
    \coordinate (p7) at (0,1);
    \coordinate (p8) at (0,0);
    \coordinate (p9) at (0,-1);
    \coordinate (q1) at (8,7);
    \coordinate (q2) at (8,6);
    \coordinate (q3) at (8,5);
    \coordinate (q4) at (8,4);
    \coordinate (q5) at (8,3);
    \coordinate (q6) at (8,2);
    \coordinate (q7) at (8,1);
    \coordinate (q8) at (8,0);
    \coordinate (x1) at (2,7);
    \coordinate (x2) at (2,6);
    \coordinate (x3) at (2,4);
    \coordinate (y1) at (4,5);
    \coordinate (y2) at (4,3);
    \coordinate (y3) at (4,2);
    \coordinate (z1) at (6,4);
    \coordinate (z2) at (6,1);
    \coordinate (z3) at (6,0);

    \def\bluepath{
    (p1) arc
    [
        start angle=90,
        end angle=270,
        x radius=1.5,
        y radius =1.5
    ] --
     (p1) arc
   [
        start angle=90,
        end angle=270,
        x radius=0.4,
        y radius =0.5
   ] --
   (p2) arc
    [
        start angle=90,
        end angle=270,
        x radius=0.6,
        y radius =1
    ] 
    }
    \draw[blue] \bluepath;
    \fill[blue,opacity=0.2,even odd rule] \bluepath;

    \def\redpath{
    (p3) arc
    [
        start angle=90,
        end angle=270,
        x radius=1.5,
        y radius =1.5
    ] --
    (p5) arc
   [
        start angle=90,
        end angle=270,
        x radius=0.4,
        y radius =0.5
   ] --
   (p3) arc
    [
        start angle=90,
        end angle=270,
        x radius=0.6,
        y radius =1
    ] 
    }
    \draw[red] \redpath;
    \fill[red,opacity=0.2,even odd rule] \redpath;

    \def \greenpath{
    (p4) arc
    [
        start angle=90,
        end angle=270,
        x radius=1.8,
        y radius =2
    ] --
    (p4) arc
   [
        start angle=90,
        end angle=270,
        x radius=1,
        y radius= 1.5
   ] --
   (p7) arc
    [
        start angle=90,
        end angle=270,
        x radius=0.4,
        y radius =0.5
    ] 
    }
    \draw[green!60!black] \greenpath;
    \fill[green!60!black,opacity=0.2,even odd rule] \greenpath;
    
    \draw (p0)--(p9);
    \draw[dashed] (p1)--(q1);
    \draw[dashed] (p2)--(q2);
    \draw[dashed] (p3)--(q3);
    \draw[dashed] (p4)--(q4);
    \draw[dashed] (p5)--(q5);
    \draw[dashed] (p6)--(q6);
    \draw[dashed] (p7)--(q7);
    \draw[dashed] (p8)--(q8);
    
 	\draw[fill] (p1) circle [radius=0.07];
 	\draw[fill] (p2) circle [radius=0.07];
 	\draw[fill] (p3) circle [radius=0.07];
 	\draw[fill] (p4) circle [radius=0.07];
 	\draw[fill] (p5) circle [radius=0.07];
 	\draw[fill] (p6) circle [radius=0.07];
        \draw[fill] (p7) circle [radius=0.07];
 	\draw[fill] (p8) circle [radius=0.07];
        \draw[fill][blue] (x1) circle [radius=0.07];
 	\draw[fill][blue] (x2) circle [radius=0.07];
 	\draw[fill][blue] (x3) circle [radius=0.07];
 	\draw[fill][red] (y1) circle [radius=0.07];
 	\draw[fill][red] (y2) circle [radius=0.07];
 	\draw[fill][red] (y3) circle [radius=0.07];
        \draw[fill][green!60!black] (z1) circle [radius=0.07];
 	\draw[fill][green!60!black] (z2) circle [radius=0.07];
        \draw[fill][green!60!black] (z3) circle [radius=0.07];

    \def \blueedge{
    (1.7,4) -- (1.7,7) -- (1.7,7) arc 
    [
        start angle=180,
        end angle=0,
        x radius=0.3,
        y radius =0.3
    ] --
    (2.3,4) -- (2.3,4) arc
    [
        start angle=360,
        end angle=180,
        x radius=0.3,
        y radius =0.3
    ]
    }
    \draw[blue] \blueedge;
    \fill[blue,opacity=0.2] \blueedge;

    \def \rededge{
    (3.7,2) -- (3.7,5) -- (3.7,5) arc 
    [
        start angle=180,
        end angle=0,
        x radius=0.3,
        y radius =0.3
    ] --
    (4.3,2) -- (4.3,2) arc
    [
        start angle=360,
        end angle=180,
        x radius=0.3,
        y radius =0.3
    ]
    }
    \draw[red] \rededge;
    \fill[red,opacity=0.2] \rededge;

    \def \greenedge{
    (5.7,0) -- (5.7,4) -- (5.7,4) arc 
    [
        start angle=180,
        end angle=0,
        x radius=0.3,
        y radius =0.3
    ] --
    (6.3,0) -- (6.3,0) arc
    [
        start angle=360,
        end angle=180,
        x radius=0.3,
        y radius =0.3
    ]
    }
    \draw[green!60!black] \greenedge;
    \fill[green!60!black,opacity=0.2] \greenedge;

    \node (q1) at (q1) [right] {$\Pi_{0,1}$};
    \node (q2) at (q2) [right] {$\Pi_{0,2}$};
    \node (q3) at (q3) [right] {$\Pi_{0,3}$};
    \node (q4) at (q4) [right] {$\Pi_{0,4}$};
    \node (q5) at (q5) [right] {$\Pi_{0,5}$};
    \node (q6) at (q6) [right] {$\Pi_{0,6}$};
    \node (q7) at (q7) [right] {$\Pi_{0,7}$};
    \node (q8) at (q8) [right] {$\Pi_{0,8}$};
    \node (e1) at (-1.5,5.5) [left] {$e_1$};
    \node (e2) at (-1.5,3.5) [left] {$e_2$};
    \node (e3) at (-1.8,2) [left] {$e_3$};
    \node (g) at (-1.5,7.5) [above] {$G$};
    \node (p) at (7,7.5) [above] {$P_0$};
    
\end{tikzpicture}\hfil}
    \caption{A visual representation of $\Pi_0$}
    \label{fig:fig1}
\end{figure}
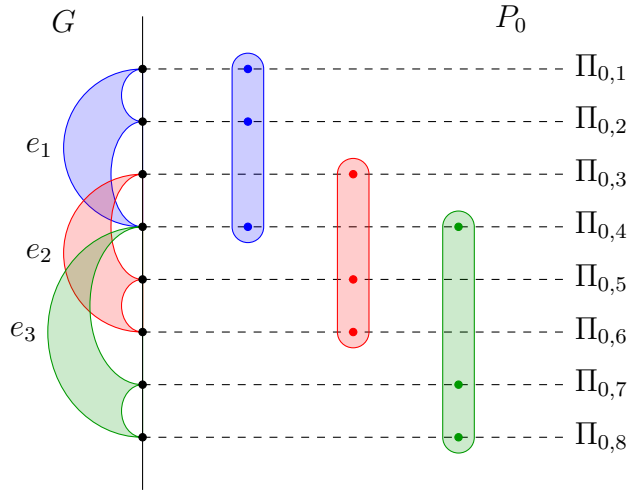

\subsection{Partite amalgamation}
In an attempt to aid the reader's orientation we would briefly like 
to mention how partite amalgamations were used in~\cite{NR79} for proving the 
existence of hypergraphs with large chromatic number and large girth (see 
also the recent survey~\cite{jarik}*{\S3.3} for more context and additional figures). 

In that argument one works with $n$-partite $k$-uniform hypergraphs instead 
of the present pictures. 
Suppose that we just have constructed some such hypergraph $\Pi$, 
and that for some index $i\in [n]$ there are $m_i$ vertices in the $i^{\mathrm{th}}$
vertex class of $\Pi$. When we have a further~$m_i$-uniform hypergraph~$\ccL$
in mind, we define the amalgamation $\Sigma=\Pi\conc\ccL$ as follows: 
The $i^{\mathrm{th}}$ vertex class of $\Sigma$ is $V(\ccL)$, each edge of $\ccL$
gets extended to its own copy of $\Pi$, distinct copies of this form are only 
allowed to intersect in the $i^{\mathrm{th}}$ vertex class, and the union of 
all these copies is the desired $n$-partite hypergraph $\Sigma$ (see 
Figure~\ref{fig:ref}). Starting with a hypergraph that looks like our 
picture zero and performing such amalgamation steps iteratively for all $i\in [n]$
we end up getting hypergraphs of large girth and chromatic number. 

Now with every picture $\Pi=(P, \psi)$ in the sense of the present work  
we can associate the partite hypergraph with vertex set~$P$ whose edges correspond 
to the combinatorial lines in~$P$. It is our intention that on the level of 
these associated hypergraphs the amalgamation of pictures introduced next 
should resemble the above construction.

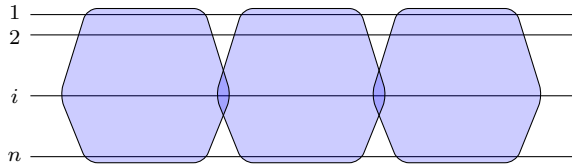
\begin{figure}[h!]
	\centering
	\begin{tikzpicture}
		
		\def\h{.27}
		\def\w{.3}
		
		\draw (-12*\w,8*\h)--(12*\w,8*\h);
		\draw (-12*\w,7*\h)--(12*\w,7*\h);
		\draw (-12*\w,4*\h)--(12*\w,4*\h);
		\draw (-12*\w,1*\h)--(12*\w,1*\h);
		
	\fill [blue, opacity=.2,rounded corners] (-2.6*\w,0.7*\h)--(2.6*\w,0.7*\h)--(3.8*\w,3.6*\w)--(2.6*\w,8.3*\h)--(-2.6*\w,8.3*\h)--(-3.8*\w,3.6*\w)--cycle;
	\fill [blue, opacity=.2,rounded corners] (9.5*\w,0.7*\h)--(4.3*\w,0.7*\h)--(3.1*\w,3.6*\w)--(4.3*\w,8.3*\h)--(9.5*\w,8.3*\h)--(10.7*\w,3.6*\w)--cycle;	
	\fill [blue, opacity=.2,rounded corners] (-9.5*\w,0.7*\h)--(-4.3*\w,0.7*\h)--(-3.1*\w,3.6*\w)--(-4.3*\w,8.3*\h)--(-9.5*\w,8.3*\h)--(-10.7*\w,3.6*\w)--cycle;
		
	\draw [rounded corners] (-2.6*\w,0.7*\h)--(2.6*\w,0.7*\h)--(3.8*\w,3.6*\w)--(2.6*\w,8.3*\h)--(-2.6*\w,8.3*\h)--(-3.8*\w,3.6*\w)--cycle;
	\draw [rounded corners] (9.5*\w,0.7*\h)--(4.3*\w,0.7*\h)--(3.1*\w,3.6*\w)--(4.3*\w,8.3*\h)--(9.5*\w,8.3*\h)--(10.7*\w,3.6*\w)--cycle;
	\draw [rounded corners] (-9.5*\w,0.7*\h)--(-4.3*\w,0.7*\h)--(-3.1*\w,3.6*\w)--(-4.3*\w,8.3*\h)--(-9.5*\w,8.3*\h)--(-10.7*\w,3.6*\w)--cycle;

	\node at (-12.7*\w,8.1*\h) {\tiny $1$};
	\node at (-12.7*\w,6.9*\h) {\tiny $2$};
	\node at (-12.7*\w,4*\h) {\tiny $i$};
	\node at (-12.7*\w,1*\h) {\tiny $n$};
		
	\end{tikzpicture}
	\caption{The hypergraph amalgamation $\Sigma=\Pi\conc\ccL$, where blue shapes 
	indicate copies of $\Pi$.}
	\label{fig:ref}
\end{figure}
 
Here are the precise details. Suppose that we have a 
picture $\Pi=(P, \psi_\Pi)$ over some $k$-uniform hypergraph $G$, 
where $P\subseteq [k]^m$. Let $x$ be a vertex 
of $G$ and suppose further that, viewing the music line~$\Pi_x$ as an alphabet 
in its own right, we are given a collection $\ccL\subseteq\ccL(\Pi_x^n)$ of 
combinatorial lines in the $n$-dimensional Hales-Jewett cube over $\Pi_x$. 
We shall now describe the construction of a pair $\Sigma=(Q, \psi_\Sigma)$ 
consisting of a set $Q\subseteq [k]^{mn}$ and a map 
$\psi_\Sigma\colon Q\lra V(G)$. This pair~$\Sigma$, which is not necessarily 
a picture again, only depends on $\Pi$ 
and $\ccL$ and will be denoted by $\Sigma=\Pi\conc\ccL$. 

Let us fix for every line $U\in \ccL$ the combinatorial 
embedding $\eta_U\colon \Pi_x\lra (\Pi_x)^n$ whose image is $U$. 
Recalling $\Pi_x\subseteq [k]^m$ we can naturally extend $\eta_U$
to a combinatorial embedding $\eta^+_U\colon [k]^m\lra [k]^{mn}$. 
Now we set $P^U=\eta_U^+(P)$ and define $\psi_{U}\colon P^U\lra V(G)$
to be the composition $\psi_{U}=\psi_\Pi\circ (\eta_U^+|_{P^U})^{-1}$. 

\begin{fact}\label{f:44}
\begin{enumerate}[labelsep=0pt, itemindent=20pt, leftmargin=0pt, label=\rmlabel]
	\item\label{it:441} $\,$ For every combinatorial line $U\in \ccL$ the 
			pair $\Pi^U=(P^U, \psi_{U})$ is a picture over $G$ with $\Pi^U_x=U$.
	\item\label{it:442} $\,\,$ If $U, V\in \ccL$ are distinct, 
			then $P^U\cap P^V=U\cap V$. 
\end{enumerate}
\end{fact}

\begin{proof}
	Beginning with~\ref{it:441} we consider an arbitrary quasiline $L\subseteq P^U$.
	Now $L'=(\eta^+_U)^{-1}[L]$ is a quasiline in $P$. Since $\Pi$ is a picture, this 
	implies that $L'$ is actually a combinatorial line and $\psi_U[L]=\psi_\Pi[L']$
	is an edge of $G$. The first statement entails that $L$ is a combinatorial line 
	as well. Finally, we have 
	$\Pi^U_x=\psi_U^{-1}(x)=(\eta^+_U\circ \psi_\Pi^{-1})(x)=\eta_U^+[\Pi_x]=U$.
	
	Proceeding with~\ref{it:442} we consider an arbitrary 
	point $z=(z(1), \dots, z(n))\in P^U\cap P^V$, where $z(1), \dots, z(n)\in P$.
	If one of the points $z(i)$ was not in $\Pi_x$, then there could be at most one line 
	$W\subseteq (\Pi_x)^n$ with $z\in \eta^+_W(P)$, whence $U=V$. 
	This argument shows $z(1), \dots, z(n)\in \Pi_x$, which in turn 
	implies $z\in \eta_U^+[\Pi_x]\cap \eta_V^+[\Pi_x]=U\cap V$. Thus we have 
	$P^U\cap P^V\subseteq U\cap V$ and owing to $U\subseteq P^U$, $V\subseteq P^V$
	the reverse inclusion is clear.  
\end{proof}
 
Now the desired pair $\Sigma=(Q, \psi_\Sigma)=\Pi\conc\ccL$ is defined by
\[
	Q=\bigcup_{U\in\ccL} P^U
	\quad\text{ and } \quad 
	\psi_\Sigma=\bigcup_{U\in\ccL} \psi_{U}\,.
\]
The pictures $\Pi^U$ occurring in Fact~\ref{f:44}\ref{it:441} are called the 
{\it standard copies} of $\Pi$ in $\Sigma$. Part~\ref{it:442} of Fact~\ref{f:44}
tells us that any two standard copies can intersect only on the music line 
$\Sigma_x$. Therefore, $\psi_\Sigma$ is indeed a function from $Q$ to $V(G)$. 

Summarising the discussion so far, one can interpret the construction 
of $\Sigma$ as follows (see Figure~\ref{fig:fig2}). First, we construct 
the music line $\Sigma_x=\bigcup\ccL$ and then for each combinatorial line 
$U\in \ccL$ we construct a standard copy~$\Pi^U$ of~$\Pi$. 
The union of all these standard copies is exactly~$\Sigma$.

\begin{figure}[h]
\centering
{\hfil \begin{tikzpicture}[scale=0.7]    
    \coordinate (l0) at (-2,0);
    \coordinate (l1) at (-2,2.5);
    \coordinate (l2) at (-2,6);
    \coordinate (r1) at (2,2.5);

    \coordinate (f0) at (6,0);
    \coordinate (f1) at (6,2.5);
    \coordinate (f2) at (6,6);
    \coordinate (g1) at (15.5,2.5);

    \coordinate (v1) at (-1,2.5);
    \coordinate (v2) at (-0.7,4);
    \coordinate (v3) at (0,5);
    \coordinate (v4) at (0.6,3.7);
    \coordinate (v5) at (1,2.5);
    \coordinate (v6) at (0.8,2);
    \coordinate (v7) at (0,0.5);
    \coordinate (v8) at (-0.6,1.9);

    \coordinate (a1) at (7,2.5);
    \coordinate (a2) at (7.3,4);
    \coordinate (a3) at (8,5);
    \coordinate (a4) at (8.6,3.7);
    \coordinate (a5) at (9,2.5);
    \coordinate (a6) at (8.8,2);
    \coordinate (a7) at (8,0.5);
    \coordinate (a8) at (7.4,1.9);

    \coordinate (b1) at (8.8,2.5);
    \coordinate (b2) at (9.1,4);
    \coordinate (b3) at (9.8,5);
    \coordinate (b4) at (10.4,3.7);
    \coordinate (b5) at (10.8,2.5);
    \coordinate (b6) at (10.6,2);
    \coordinate (b7) at (9.8,0.5);
    \coordinate (b8) at (9.2,1.9);

    \coordinate (c1) at (12.4,2.5);
    \coordinate (c2) at (12.7,4);
    \coordinate (c3) at (13.4,5);
    \coordinate (c4) at (14,3.7);
    \coordinate (c5) at (14.4,2.5);
    \coordinate (c6) at (14.2,2);
    \coordinate (c7) at (13.4,0.5);
    \coordinate (c8) at (12.8,1.9);

    \coordinate (d1) at (10.6,2.5);
    \coordinate (d2) at (10.9,4);
    \coordinate (d3) at (11.6,5);
    \coordinate (d4) at (12.2,3.7);
    \coordinate (d5) at (12.6,2.5);
    \coordinate (d6) at (12.4,2);
    \coordinate (d7) at (11.6,0.5);
    \coordinate (d8) at (11,1.9);
    
    \draw (l0)--(l2);
    \draw (f0)--(f2);
    \draw[dashed] (l1)--(r1);
    \draw[dashed] (f1)--(g1);
    
    \draw[fill][gray, opacity=0.4] (0,2.5) ellipse (1 and 0.2);
    \draw (0,2.5) ellipse (1 and 0.2);
    
    \draw[fill][gray, opacity=0.4] (8,2.5) ellipse (1 and 0.2);
    \draw (8,2.5) ellipse (1 and 0.2);
    
    \draw[fill][gray, opacity=0.4] (9.8,2.5) ellipse (1 and 0.2);
    \draw (9.8,2.5) ellipse (1 and 0.2);

    \draw[fill][gray, opacity=0.4] (11.6,2.5) ellipse (1 and 0.2);
    \draw (11.6,2.5) ellipse (1 and 0.2);

    \draw[fill][green, opacity=0.4] (13.4,2.5) ellipse (1 and 0.2);
    \draw (13.4,2.5) ellipse (1 and 0.2);

    \draw (v1) to [closed, curve through={(v2) (v3) (v4) (v5) (v6) (v7) (v8)} ] (v1);
    \draw (a1) to [closed, curve through={(a2) (a3) (a4) (a5) (a6) (a7) (a8)} ] (a1);
    \draw (b1) to [closed, curve through={(b2) (b3) (b4) (b5) (b6) (b7) (b8)} ] (b1);
    \draw (d1) to [closed, curve through={(d2) (d3) (d4) (d5) (d6) (d7) (d8)} ] (d1);
    \draw[green!60!black] (c1) to [closed, curve through={(c2) (c3) (c4) (c5) (c6) (c7) (c8)} ] (c1);

    \draw[fill] (l1) circle [radius=0.07];
    \draw[fill] (f1) circle [radius=0.07];
 	
    \draw[->] (3,2.5)--(4,2.5);

    \node (l1) at (l1) [left] {$\Pi_x$};
    \node (f1) at (f1) [left] {$\Sigma_x$};
    \node (v3) at (v3) [above] {$P$};
    \node (pL) at (14.4,3) [right][green!60!black] {$U\cong \Pi_x$};
    \node (c3) at (c3) [above][green!60!black] {$P^U\cong P$};
    \node (q6) at (8,6) [right] {$Q$};
                                
\end{tikzpicture}\hfil}
    \caption{The partite amalgamation $\Sigma=\Pi\conc \ccL$.}
    \label{fig:fig2}
\end{figure}

In general, the partite amalgamation $\Sigma=\Pi\conc \ccL$ does not necessarily 
create a new picture, because there could be ``unintended'' quasilines in $Q$ 
whose points belong to several distinct standard copies of $\Pi$. 
The main result of this subsection shows how we will avoid this situation
in the future. 
  
\begin{proposition}\label{prop:amal}
	Let $\Pi=(P, \psi_\Pi)$ be a picture over a $k$-uniform 
	hypergraph $G$, where $k\ge 3$ and if $k=3$, then $G$ is 
	$K_4^{(3)-}$-free. If $x$ denotes a vertex of $G$ and  
	$\ccL\subseteq \ccL(\Pi_x^n)$ is a collection of 
	combinatorial lines containing neither tripods nor triangles,	
	then $\Pi\conc\ccL$ is again a picture over $G$. 
\end{proposition}

\begin{proof}	Continuing our earlier notation we again suppose that $P\subseteq [k]^m$
	and we write $\Sigma=(Q, \psi_\Sigma)$ for the pair $\Sigma=\Pi\conc\ccL$.
	The main task is to establish the following statement. 
		\begin{equation}\label{eq:qlip}
	\text{\it
       For every quasiline $L\subseteq Q$ there is some $U\in \ccL$ 
       such that $L\subseteq P^U$.}
	\end{equation}
	
	In other words, the only quasilines in $Q$ are those contained in standard 
	copies of~$\Pi$. Assuming for the moment that this holds, 
	it follows as in the proof of Fact~\ref{f:44}\ref{it:441} that all quasilines 
	$L\subseteq Q$ are combinatorial lines 
	projecting onto edges of $G$, i.e., that $\Sigma$ is indeed a picture.  
	Thus, it remains to show~\eqref{eq:qlip}. 
		
	To this end we write $L=\{\ell_1,\dots,\ell_k\}$ and  
	$\ell_i=\bigl(\ell_i(1), \dots, \ell_i(n)\bigr)$ for every $i\in [k]$, 
	where $\ell_i(1), \dots, \ell_i(n)\in P$. Since $L$ is a quasiline, 
	for every $s\in [n]$ the set 
		\[
		L_s=\bigl\{\ell_1(s), \dots, \ell_k(s)\bigr\}
	\]
		either consists of a single element, or it is a quasiline contained in~$P$.
	In the latter case, $L_s$ is actually a combinatorial line, because $\Pi$ 
	is a picture. 	
	Owing to the construction of $\Sigma$ there exist 
	lines $U_1, \dots, U_k\in\ccL$ such that $\ell_i\in\Pi^{U_i}$ 
	for every $i\in [k]$ and there are points $c_1, \dots, c_k\in P$ 
	such that $\ell_i=\eta^+_{U_i}(c_i)$. For clarity we point out that 
	for $\ell_i\not\in \Sigma_x$ the pair $(U_i, c_i)$ is uniquely determined 
	by $\ell_i$. On the other hand, if~$\ell_i\in \Sigma_x$, then there can be 
	several legitimate choices for $U_i$, but $c_i\in \Pi_x$ will then necessarily 
	be true.
	
	In general, we have 
		\begin{equation}\label{eq:2214}
		\ell_i(s)\in\Pi_x\cup\{c_i\}
		\text{ for all $i\in [k]$ and $s\in [n]$,}
	\end{equation}
	whence $L_s\subseteq \Pi_x\cup\{c_1, \dots, c_k\}$. The remainder of the proof 
	exploits heavily that the set $\Pi_x\cup\{c_1, \dots, c_k\}$ can contain only
	very few combinatorial lines.  
	
	\begin{claim}\label{clm:1741}
		Every combinatorial line $K\subseteq \Pi_x\cup\{c_1, \dots, c_k\}$
		contains at most one point from $\Pi_x$ and at least $k-1$ points 
		from $\{c_1, \dots, c_k\}\sm \Pi_x$.
	\end{claim}
	
	\begin{proof}
		Since $\psi_\Pi$ projects $K$ onto an edge of $G$ while all points in $\Pi_x$
		are projected to the same vertex $x$, we have $|K\cap \Pi_x|\le 1$. 
		Due to $|K|=k$ the second assertion follows. 
	\end{proof}

	Let $C=\{s\in [n]\colon |L_s|=1\}$ be the set of coordinates where the points 
	of our quasiline~$L$ agree. So for every $c\in C$ there is some point $\ell(c)$
	such that $\ell(c)=\ell_1(c)=\dots =\ell_k(c)$. Because of $|L|=k$ we have 
		\begin{equation}\label{eq:Cn}
		C\ne [n]\,.
	\end{equation}
	
	Assume for the sake of contradiction that $\ell(c^\star)\not\in\Pi_x$ holds 
	for some $c^\star\in C$. In view of~\eqref{eq:2214} this implies 
	$c_1=\dots=c_k=\ell(c^\star)$. Now Claim~\ref{clm:1741} shows that the 
	set $\Pi_x\cup\{c_1, \dots, c_k\}$ contains no combinatorial lines, which 
	in turn leads to $C=[n]$. This contradiction to~\eqref{eq:Cn} establishes 
		\begin{equation}\label{eq:lcx}
		\ell(c)\in \Pi_x 
		\quad \text{ for all $c\in C$}. 
	\end{equation}
		
	Let us now pick an arbitrary coordinate $s^\star\in [n]\sm C$. 
	Due to~\eqref{eq:2214} and Claim~\ref{clm:1741}
	we may assume, without loss of generality, that $\ell_{s^\star}(i)=c_i\not\in \Pi_x$
	holds for every $i\in [k-1]$. Concerning the point $a=\ell_{s^\star}(k)$, however, 
	we know nothing more than that it is in $\Pi_x\cup\{c_k\}$. 
	We shall show later that the set 
		\[
		S=\bigl\{s\in [n]\sm C\colon 
		\ell_{s}(i)=\ell_{s^\star}(i) \text{ for every $i\in [k]$}\bigr\}
	\]
		is equal to $[n]\sm C$. Assuming for the moment that this is true, the proof 
	of~\eqref{eq:qlip} can be completed as follows. Let $U\subseteq \Pi_x^n$
	be the combinatorial line whose set of moving coordinates is $S$ and which 
	takes the values $\ell(c)$ on its constant coordinates $c\in C$. The definition 
	of $S$ discloses $L=\eta_U^+[L_{s^\star}]$. 
	Furthermore~\eqref{eq:lcx} and
	$c_1\not\in\Pi_x$ imply $U=U_1$ and thus we have $U\in \ccL$. 
	So altogether $U$ is the line required by~\eqref{eq:qlip}. 
	
	In the remainder of the argument we shall show that the assumption $S\ne [n]\sm C$
	leads to the contradiction that either $\ccL$ contains a tripod or a triangle,
	or $k=3$ and $G$ contains a $K_4^{(3)-}$.
	Considering the nonempty set $T=[n]\sm (C\dcup S)$ we distinguish two cases.
	 	
	\smallskip
	
	{\hskip2em \it First Case. We have $L_{t}=L_{s^\star}$ for every $t\in T$.}
	
	\smallskip
	
	Pick an arbitrary coordinate $t^\star\in T$.  
	Roughly speaking, the equality $L_{t^\star}=L_{s^\star}$ means that the 
	lines~$L_{s^\star}$ and~$L_{t^\star}$
	contain the same points, but not ``in the same order''. By the definition of $S$,
	there needs to exist some $i\in [k-1]$ such 
	that $\ell_{s^\star}(i)\ne \ell_{t^\star}(i)$ and without loss of generality we 
	can assume that this happens for $i=1$. 
	So $\ell_{t^\star}(1)\ne \ell_{s^\star}(1)=c_1$.
	
	Due to $\ell_{t^\star}(1)\in (\Pi_x\cup\{c_1\})\cap \{c_2, \dots, c_{k-1}, a\}$
	we have $\ell_{t^\star}(1)=a\in \Pi_x$. Now~\eqref{eq:2214} tells 
	us $\ell_{t^\star}(i)=c_i$ for every $i\in [2, k-1]$ and together with
	$L_{t^\star}=L_{s^\star}$ we obtain $\ell_{t^\star}(k)=c_1$. 
	In view of~\eqref{eq:2214} and $c_1\not\in \Pi_x$ this shows $c_1=c_k$ (see
	Figure~\ref{fig:A}). 
	
	\begin{figure}[h]
\centering
	\begin{tikzpicture}
	\def\w{1.7}
	\def\h{2.5}
	
	\coordinate (a) at (-.4*\w, -.4*\h);
	\coordinate (c) at (0, .4*\h);
	\coordinate (b) at ($(a)!.27!(c)$);
	
	\draw (a) -- (c);
	\draw [dotted, shorten >= 15] ($(b)+(0,.35)$)--($(c)+(0,.35)$);
	
	\draw (\w,0) [out=90, in=0] to (0,\h) [out=180, in=90] to (-\w, 0) [out=-90, in=180] to (0,-\h) [out = 0, in =-90] to (\w, 0);
	
	\draw (-.93*\w, -.4*\h)--(.93*\w, -.4*\h);
	
	\foreach \i in {a,b,c} \fill (\i) circle (1.5pt);
	
	\node [below] at (a) {\small $a$};
	\node [right] at (b) {\small $c_1=c_k$};
	\node [right] at (c) {\small $c_{k-1}$};
	\node at (-1.1*\w, -.4*\h) {$\Pi_x$};
	\end{tikzpicture}
	\caption{The line $L_{s^\star}\subseteq P$.}
	\label{fig:A}
\end{figure}
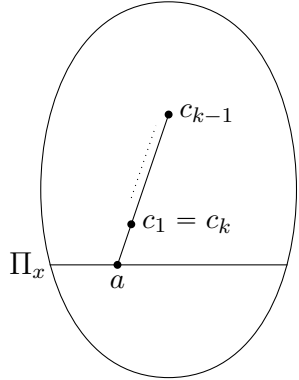

	We contend that 
		\begin{equation}\label{eq:tt}
		\ell_t(i)=\ell_{t^\star}(i)
	\end{equation}
		holds for all $t\in T$ and all $i\in [k]$. To see this we fix any $t\in T$
	and recall that $L_t=L_{s^\star}$. For every $i\in [2, k-1]$ the point $c_i$
	needs to appear somewhere in $L_t$, but due to~\eqref{eq:2214} only $\ell_t(i)=c_i$
	is possible. This leaves us with $\{\ell_t(1), \ell_t(k)\}=\{a, c_1\}$ and 
	in view of $t\not\in S$ we obtain $\ell_t(1)=a$ and $\ell_t(k)=c_1$, 
	which proves~\eqref{eq:tt}. 
	
	Thereby we have determined the points $\ell_1, \dots, \ell_k$ completely
	and we arrive at the following description of the lines $U_1$, $U_2$, $U_k$.
	
	\begin{center}
	\begin{tabular}{c|c|c|c}
	& $C$ & $S$ & $T$ \\
	\hline
	$U_1$ & constant $\ell(c)$ & moving & constant $a$ \\
	\hline
	$U_2$ & constant $\ell(c)$ & moving & moving \\
	\hline
	$U_k$ & constant $\ell(c)$ & constant $a$ & moving
	\end{tabular}
	\end{center}	 
	
	All three lines pass through $\eta_{U_1}^+(a)=\eta_{U_2}^+(a)=\eta_{U_k}^+(a)$.
	Thus they form a tripod in $\ccL$, contrary to our hypothesis.
	
	\smallskip
	
	{\hskip2em \it Second Case. Some $t^\star\in T$ 
	satisfies $L_{t^\star}\ne L_{s^\star}$.}
	
	\smallskip
	
	The distinct combinatorial lines $L_{s^\star}$ and $L_{t^\star}$
	can intersect in at most one point. On the other hand, Claim~\ref{clm:1741}
	tells us that both of them contain at least $k-1$ points 
	from $\{c_1, \dots, c_k\}\sm \Pi_x$. For these reasons we have $k=3$,
	$\Pi_x\cap\{c_1, c_2, c_3\}=\vn$ and, without loss of generality,
	$L_{s^\star}=\{a, c_1, c_2\}$, $L_{t_\star}=\{b, c_1, c_3\}$,
	where $a, b\in \Pi_x$ are distinct (see Figure~\ref{fig:B}). 
	
	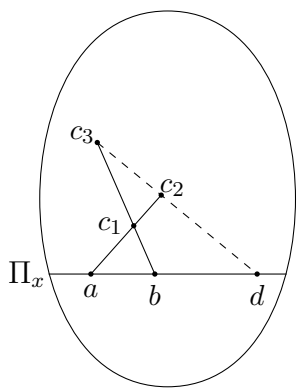
\begin{figure}[h]
\centering
	\begin{tikzpicture}
	\def\w{1.7}
	\def\h{2.5}

	\coordinate (a) at (-.6*\w, -.4*\h);
	\coordinate (b) at (-.1*\w, -.4*\h);
	\coordinate (d) at (.7*\w, -.4*\h);
	\coordinate (c3) at (-.55*\w, .3*\h);
	\coordinate (c1) at (-.265*\w, -.143*\h);
	\coordinate (c2) at ($(c3)!.4!(d)$);
	
	\draw (a) -- (c2);
	\draw (b) -- (c3);
	\draw [dashed] (c3)--(d);
	
	\draw (\w,0) [out=90, in=0] to (0,\h) [out=180, in=90] to (-\w, 0) [out=-90, in=180] to (0,-\h) [out = 0, in =-90] to (\w, 0);
	
	\draw (-.93*\w, -.4*\h)--(.93*\w, -.4*\h);
	
	\foreach \i in {a,b,d, c1, c2, c3} \fill (\i) circle (1pt);
	
	\node [below] at (a) {\small $a$};
	\node [below] at (b) {\small $b$};
	\node [below] at (d) {\small $d$};
	\node [left] at (c1) {\small $c_1$};
	\node at ($(c3)+(-.2,.1)$) {\small $c_3$};
	\node at ($(c2)+(.15,.1)$) {\small $c_2$};
	\node at (-1.1*\w, -.4*\h) {$\Pi_x$};
	
	\end{tikzpicture}
	\caption{The lines $L_{s^\star}, L_{t^\star}\subseteq P$.}
	\label{fig:B}
\end{figure} 	
	If there existed a third line $L_\bullet\subseteq \Pi_x\dcup\{c_1, c_2, c_3\}$,
	it had to be of the form  
	$L_\bullet=\{d, c_2, c_3\}$ for some $d\in \Pi_x$, but then the projections 
	$\psi_\Pi[L_{s^\star}]$, $\psi_\Pi[L_{t^\star}]$, $\psi_\Pi[L_\bullet]$
	formed a $K_4^{(3)-}$ in $G$, contrary to our assumptions. 
	
	This proves that $L_{s^\star}$ and $L_{t^\star}$
	are the only lines in $\Pi_x\dcup\{c_1, c_2, c_3\}$. It is now easy to see that 
	every $t\in T$ satisfies $\ell_t(1)=c_1$, $\ell_t(2)=b$, and $\ell_t(3)=c_3$,
	which in turn yields the following description of the lines $U_1$, $U_2$, and $U_3$.
	
	\begin{center}
	\begin{tabular}{c|c|c|c}
	& $C$ & $S$ & $T$ \\
	\hline
	$U_1$ & constant $\ell(c)$ & moving & moving \\
	\hline
	$U_2$ & constant $\ell(c)$ & moving & constant $b$ \\
	\hline
	$U_3$ & constant $\ell(c)$ & constant $a$ & moving
	\end{tabular}
	\end{center}
	
	\begin{figure}[h]
\centering
	\begin{tikzpicture}
	\coordinate (a) at (-4, 0);
	\coordinate (b) at (0,2);
	\coordinate (c) at (.4, -2);
	\coordinate (ab) at ($(a)!.5!(b)$);
	\coordinate (ac) at ($(a)!.5!(c)$);
	\coordinate (bc) at ($(b)!.5!(c)$);
	
	\draw [shorten <=-15, shorten >=-15](a)--(b);
	\draw [shorten <=-15, shorten >=-15](a)--(c);
	\draw [shorten <=-15, shorten >=-15](c)--(b);
	
	\foreach \i in {a,b,c}
		\fill (\i) circle (1.5pt);
		
	\node [above] at (ab) {$U_2$};
	\node [right] at (bc) {$U_1$};
	\node [below] at (ac) {$U_3$};
	\node at ($(a)+(-1.8,0)$) {\small $\eta_{U_2}^+(a)=\eta_{U_3}^+(b)$};
	\node at ($(b)+(1.5,-.1)$) {\small $\eta_{U_1}^+(b)=\eta_{U_2}^+(b)$};
	\node at ($(c)+(1.5,.3)$) {\small $\eta_{U_1}^+(a)=\eta_{U_3}^+(a)$};
	
	\end{tikzpicture}
	\caption{A triangle in $\ccL$.}
	\label{fig:C}
\end{figure} 	 
	Therefore, any two of the three lines $U_1$, $U_2$, $U_3$ intersect but due to 
	$a\ne b$ they do not pass through a common point (see Figure~\ref{fig:C}).
	This contradicts the assumption that $\ccL$ contains no triangle.
\end{proof}

\subsection{The construction of \texorpdfstring{$\cX(k, r, \mu)$}{X(k, r, mu)}}
\label{subsec:43}

This subsection is devoted to the proof of Theorem~\ref{thm:main3}.
Recall that we are given two integers $k\geq 3$, $r\geq 1$, and a real 
$\mu\in (0, \frac{k-1}{k})$. Theorem~\ref{thm:graphpisier} delivers a
$k$-uniform hypergraph $G$ with $\chi(G)>r$ which has the $\mu$-fractional 
property. In the special case $k=3$ Lemma~\ref{lem:k43-} allows us to assume, 
additionally, that~$G$ is $K_4^{(3)-}$-free. For notational simplicity we can 
suppose $V(G)=[q]$ for some natural number~$q$. 

Let $\Pi_0=(P_0, \psi_0)$ denote the picture zero over $G$ provided by 
Lemma~\ref{lem:Pic0}. Starting with~$\Pi_0$ we shall define recursively 
a sequence $(\Pi_i)_{i\le q}$ of pictures over $G$. These pictures will be 
written in the form $\Pi_i=(P_i, \psi_i)$, where $P_i\subseteq [k]^{m_i}$ 
for some dimension $m_i$, and their music lines will be denoted 
by $\Pi_{i, j}=\psi_i^{-1}(j)$ for all $j\in [q]$. 
As the proof of Lemma~\ref{lem:Pic0} shows, picture zero 
can be assumed to have the dimension $m_0=2|E(G)|$, but this fact is of 
no importance to what follows. The remaining terms of the 
sequence $(m_i)_{i\le q}$ will be defined together with the corresponding 
pictures. 

Suppose now that for some $i\in [q]$ we have just constructed the 
picture $\Pi_{i-1}$. Theorem~\ref{thm:tripod} applied to the music 
line $\Pi_{i-1, i}$ here in place of $A$ there yields for some 
dimension~$n_i$ a collection $\ccL_i\subseteq \ccL(\Pi_{i-1, i}^{n_i})$ 
of combinatorial lines containing neither tripods nor triangles such 
that $\chi(\ccL_i)>r$. Owing to Proposition~\ref{prop:amal}
the structure $\Pi_i=\Pi_{i-1}\conc\ccL_i$ is again a picture over~$G$. 
For definiteness we point out that $P_i\subseteq [k]^{m_i}$ holds for 
$m_i=m_{i-1}n_i$. This concludes the explanation how we move from one 
picture~$\Pi_{i-1}$ of our sequence to the subsequent one. 

It will turn out that the final picture, or more precisely the set 
$\cX=\cX(k, r, \mu)=P_q$, has the properties described in 
Theorem~\ref{thm:main3}. As usual in arguments by partite construction, 
our stipulations unfold as follows. 

\begin{claim}\label{cl:partiteramsey}
	If $\gamma$ denotes an $r$-colouring of $P_q$, then for  
	every nonnegative $i\le q$ there exist a combinatorial 
	embedding $\eta\colon [k]^{m_i}\lra [k]^{m_q}$
	with $\eta[P_i]\subseteq P_q$ and colours 
	$\rho_{i+1}, \dots, \rho_q\in [r]$ such that 
	$(\gamma\circ\eta)(x)=\rho_j$ holds whenever 
	$x\in \Pi_{i, j}$ and $j\in (i, q]$.
\end{claim}

\begin{proof}
	We proceed by backwards induction on $i$. The statement is vacuously true 
	for~$i=q$. Suppose now that Claim~\ref{cl:partiteramsey} holds for some 
	positive $i\le q$ and that a colouring $\gamma\colon P_q\lra [r]$
	is given. The induction hypothesis shows that there are a combinatorial 
	embedding $\oleta\colon [k]^{m_i}\lra [k]^{m_q}$ with $\oleta[P_i]\subseteq P_q$
	and colours $\rho_{i+1}, \dots, \rho_q\in [r]$ such 
	that $(\gamma\circ\oleta)(x)=\rho_j$ whenever $x\in \Pi_{i, j}$ 
	and $j\in (i, q]$. 
	
	Notice that $\olg=\gamma\circ\oleta$ is an $r$-colouring of $P_i$
	and, hence, of $\Pi_{i, i}=\bigcup\ccL_i$. By our choice of the line 
	system $\ccL_i$ some combinatorial line $U\in \ccL_i$ is monochromatic 
	with respect to $\olg$, say with colour $\rho_i$. Due to the construction 
	of $\Pi_i=\Pi_{i-1}\conc\ccL_i$ the picture $\Pi_i$ contains a standard 
	copy $\Pi_{i-1}^U$ of $\Pi_{i-1}$ whose underlying set $P_{i-1}^U$
	is given by $P_{i-1}^U=\eta^+_U[P_{i-1}]$, 
	where $\eta^+_U\colon [k]^{m_{i-1}}\lra [k]^{m_i}$ is a combinatorial 
	embedding such that $\eta^+_U[\Pi_{i-1, i}]=U$. 
	
	We contend that the combinatorial embedding $\eta=\oleta\circ \eta^+_U$
	from $[k]^{m_{i-1}}$ to $[k]^{m_q}$ and the colours $\rho_i, \dots, \rho_q$ 
	have the desired properties. To confirm this, we consider any point 
	$x\in \Pi_{i-1, j}$, where $j\in [i, q]$. 
	Due to  
	$\gamma\circ\eta=\gamma\circ\oleta\circ \eta^+_U=\olg\circ \eta^+_U$
	we need to show $(\olg\circ \eta^+_U)(x)=\rho_j$. In the special case  
	$j=i$ this follows from $\eta^+_U(x)\in U$ and for $j\in (i, q]$ we 
	can appeal to $\eta^+_U(x)\in\Pi_{i, j}$ combined with the choice of 
	$\rho_j$. 
\end{proof}
 
We are now ready to prove that $\cX=P_q$ satisfies clause~\ref{it:17i}
of Theorem~\ref{thm:main3}. Given a colouring $\gamma\colon P_q\lra [r]$ the 
case $i=0$ of Claim~\ref{cl:partiteramsey} delivers a combinatorial embedding 
$\eta\colon [k]^{m_0}\lra [k]^{m_q}$ with $\eta[P_0]\subseteq P_q$ 
and colours $\rho_1, \dots, \rho_q\in [r]$ such that $(\gamma\circ\eta)(x)=\rho_j$ 
whenever $j\in [q]$ and $x\in \Pi_{0, j}$. Due to $\chi(G)>r$ 
there is an edge $e$ of $G$ that is monochromatic with respect to 
the $r$-colouring $i\longmapsto \rho_i$ of $V(G)=[q]$.
Next, by Lemma~\ref{lem:Pic0} there is 
a combinatorial line $L_e\subseteq P_0$ with $\psi_0[L_e]=e$. 
Now $\eta[L_e]$ is a combinatorial line in $P_q$ all of whose 
points have the same colour as $e$. 

It remains to address part~\ref{it:17ii} of Theorem~\ref{thm:main3}. 
For this purpose we consider the $k$-uniform hypergraph $H$ with vertex 
set $V(H)=\cX=P_q$ whose edges correspond to the combinatorial 
lines $L\subseteq P_q$. 
Since $\Pi_q=(P_q, \psi_q)$ is a picture, we could equivalently
say that the edges of $H$ are the quasilines in $P_q$. 
Moreover, $\psi_q$ is a hypergraph
homomorphism from~$H$ to~$G$. As~$G$ has the $\mu$-fractional property, 
Lemma~\ref{cl:hereditary} implies that $H$ has this property, too. 
In particular, every set $\cY\subseteq \cX$ has a subset $\cZ\subseteq \cY$ 
of size $|\cZ|\geq \mu|\cY|$ which is independent in $H$ and, therefore, contains 
no quasilines. This completes the proof of Theorem~\ref{thm:main3}.

\section{Concluding remarks}\label{sec:remarks}
A $k$-tuple $(x_1, \dots, x_k)$ of natural numbers forms a (possibly degenerate)
arithmetic progression of length $k$ if and only if it solves the homogeneous
system of linear equations
\begin{align}\label{aeq:eq20}
x_i-2x_{i+1}+x_{i+2}=0, \quad \text{where $i=1, \dots, k-2$}.
\end{align}
Thus van der Waerden's theorem and Szemer\'{e}di's theorem can be regarded as 
Ramsey theoretic statements on the solutions of \eqref{aeq:eq20}. Similar results 
have also been studied for more general systems of equations, and one may  
wonder for which systems the natural analogue of Theorem~\ref{thm:main1} holds.

Given a matrix $A\in \ZZ^{m\times n}$ with integer coefficients, 
the system of homogeneous linear equations $A\bx=0$ is called 
\textit{partition regular} if 
for every finite colouring of $\NN$ there exists a monochromatic solution 
$\bx=(x_1, \dots, x_n)^T$ of the system. Examples of partition regular 
systems include 
the single equation $x_1+x_2=x_3$ (Schur's theorem) and arithmetic 
progressions (van der 
Waerden's theorem). A full characterisation of partition regularity was 
obtained by Rado~\cites{R43,D89}.

Similarly, a homogeneous linear system $A\bx=0$ is said to be {\it density regular}
if for every subset $X\subseteq \NN$ of positive upper density there is a 
solution $\bx\in X^n$ which consists of $n$ distinct integers. Density regularity
implies partition regularity (by focusing on the densest colour class), but 
not the other way around. For instance, Schur's equation $x_1+x_2=x_3$ is partition regular but not density regular (as it has no solution with three odd numbers). 
Frankl, Graham, and the second author~\cite{FGR88} gave an explicit 
characterisation of density regular systems.  

It would be interesting to determine for which systems of linear equations there 
exists a version of Theorem~\ref{thm:main1}.

\begin{question}\label{q:last}
Given a system of linear equations $A\bx=0$ with $A\in \ZZ^{m\times n}$ do there
exist a set of natural numbers $X\subseteq \NN$ and a real number $\epsilon>0$ such 
that
\begin{enumerate}[label=\rmlabel]
    \item for every finite colouring of $X$ there is a monochromatic solution 
    of $A\bx=0$
    \item and every finite set $Y\subseteq X$ has a subset $Z\subseteq Y$ 
    	with $|Z|\geq \epsilon |Y|$ not containing a non-trivial solution of $A\bx=0$?
\end{enumerate}
\end{question}

We conjecture that for density regular systems the answer is affirmative. 
An interesting special case is offered by the single 
equation $x_1+\dots+x_h=y_1+\dots+y_h$. Sets without 
nontrivial solutions to this equation, called $B_h$-sets, 
have been studied intensively in the literature. Guided by 
Paul~Erd\H{o}s, the last two authors proved together with Ne\v{s}et\v{r}il 
that Question~\ref{q:last} has a positive answer for $B_h$-sets 
(see~\cite{NRS}*{Theorem~1.2}). As $h$ tends to infinity, their value of $\eps$
converges to zero very rapidly. The girth Ramsey theorem~\cite{girth}
implies that one can take $\eps=1/4$ uniformly in $h$, as we shall explain 
in forthcoming work. 
  
\subsection*{Acknowledgement}
We are grateful to {\sc Joanna Polcyn} for her generous help with the figures.
Moreover, we would like to thank the referees for reading our article very 
carefully.

\end{document}